\documentclass[11pt]{article}
\usepackage{hyperref}
\usepackage{mathtools}
\usepackage{dsfont}
\usepackage{verbatim}
\mathtoolsset{showonlyrefs}
\usepackage{mathrsfs}
\usepackage{todonotes}

\usepackage{geometry}                
\usepackage{graphicx}
\usepackage{amssymb,amsmath,amsthm}
\usepackage{epstopdf}
\usepackage{enumerate}
\usepackage{enumitem}
\usepackage[english]{babel}
\usepackage{esint}
\usepackage{subcaption}
\usepackage{tikz}
\usepackage{xcolor}
\usepackage{pgfplots}
\usetikzlibrary{calc}
\usetikzlibrary{calc, arrows.meta}
\newcommand{\LM}[2]{\hbox{\vrule width.4pt \vbox to#1pt{\vfill
\hrule width#2pt height.4pt}}}
\newcommand{\LLL}{{\mathchoice {\>\LM{7}{5}\>}{\>\LM{7}{5}\>}{\,\LM{5}{3.5}\,}{\,\LM{3.35}{2.5}\,}}}

\newtheorem{theorem}{Theorem}[section]

\newtheorem{lemma}[theorem]{Lemma}
\newtheorem{proposition}[theorem]{Proposition}

\theoremstyle{definition}

\newtheorem{remark}[theorem]{Remark}

\def\R{\mathbb{R}}
\def\Z{\mathbb{Z}}

\def\1{\mathds{1}}
\def\N{\ensuremath{\mathbb{N}}}

\let\phi\varphi
\let\epsilon\varepsilon

\DeclarePairedDelimiterX{\inp}[2]{\langle }{\rangle }{#1 \, , #2}

\newcommand{\interior}[1]{%
  {\kern0pt#1}^{\mathrm{o}}%
}

\newcommand{\glim}{\mathop{\Gamma\text{-}\lim}\limits}

\def\R{\mathbb{R}}
\def\Z{\mathbb{Z}}

\def\1{\mathds{1}}
\def\N{\ensuremath{\mathbb{N}}}

\title{Homogenization and averaging in the Bour\-gain--Brezis--Mironescu limit}
\author{Andrea Braides\\
\small Department of Mathematics, University of Rome Tor Vergata\\
\small via della ricerca scientifica 1, 00133 Rome, Italy
\\ \\
 Stefano Mannella \ and \ Alec Jacopo Almo Schiavoni Piazza
 \\
 \small SISSA, via Bonomea 265, 34146 Trieste, Italy
 }
\date{}      

\begin{document}
\maketitle

\begin{abstract}
We study the simultaneous asymptotic behaviour of homogenization and the Bour\-gain--Brezis--Mironescu limit for quadratic fractional energies with periodic coefficients. The functionals depend on two small parameters: the period of the coefficients and the fractional exponent approaching the local limit. We prove that the $\Gamma$-limit is completely determined by the relative scaling of these parameters. In the two extreme regimes, one recovers the expected separation of scales, yielding either the homogenized Dirichlet energy or the energy with averaged coefficients. At the critical logarithmic scaling, however, the two mechanisms coexist, and the limit is a convex combination of the homogenized and averaged energies, with weights depending explicitly on the relative scaling of the two asymptotic processes. The proof combines the homogenization analysis of the short-range interactions with a blow-up argument showing that long-range interactions are simply averaged using a Riemann--Lebesgue-type argument.

\smallskip{\bf Keywords:} homogenization; Bourgain--Brezis--Mironescu formula; fractional Sobo\-lev spaces; nonlocal functionals; Gagliardo seminorm; averaging.

\smallskip{
\bf MSC 2020:} 49J45, 35B27, 
35R11, 74Q05
\end{abstract}

\section{Introduction}
In this work, we study the interaction between two fundamental asymptotic processes in the Calculus of Variations: homogenization and the Bourgain--Brezis--Mironescu limit for fractional Sobolev energies. Although each of these limits is now well understood, their simultaneous behaviour exhibits a genuinely new phenomenon, whose nature depends on the relative scaling between the microscopic periodicity and the fractional parameter.

The Bourgain--Brezis--Mironescu theorem \cite{BBM,ponce} characterizes Sobolev spaces through the asymptotic behaviour of suitable nonlocal energies. In the quadratic case, the normalized Gagliardo seminorms
\[
(1-s)\int_\Omega\int_\Omega \frac{|u(x)-u(y)|^2}{|x-y|^{d+2s}}dxdy,
\]
$\Gamma$-converge, as $s\to 1^-$, to a multiple of the Dirichlet energy. More precisely, the limit is $\frac{\sigma_{d-1}}{2d}$ times the Dirichlet energy, where $\sigma_{d-1}$ is the $d-1$-surface measure of the unit sphere in $\mathbb R^d$. Thus, the result provides a variational characterization of  $H^1(\Omega)$. This result extends to $W^{
s,p}$
 spaces and to a broad class of kernels \cite{Davo,GS}, and has found numerous applications in nonlocal variational problems (see e.g.~\cite{ADP,SV,PV,Solci}). 

 Although many extensions of the Bourgain--Brezis--Mironescu theorem are now available, much less is known about its interaction with other singular limits involving microscopic scales. In particular, a complete description of its interplay with homogenization is not available. Homogenization and the Bourgain--Brezis--Mironescu limit both replace a complicated energy by an effective local one, but through completely different mechanisms. The former accounts for microscopic oscillations of the coefficients, whereas the latter concentrates the interaction kernel at infinitesimal scales. 
 The question is then how these two asymptotic mechanisms interact when they take place simultaneously. As we show, neither mechanism dominates, in general.

For the reader's convenience, we briefly recall the variational formulation of periodic homogenization. This is a prototypical example of the nontrivial effects of oscillations in differential equations and the Calculus of Variations. In the simplest deterministic periodic (isotropic and quadratic) variational form, it consists of studying the asymptotic behaviour of minimum problems involving integrals 
\[
\int_\Omega a\left(\frac{x}\epsilon\right)|\nabla u|^2dx
\]
as $\epsilon\to 0$.
The coefficient $a$ is $1$-periodic, bounded, and bounded away from zero.
This analysis was one of the examples that led De Giorgi to the definition of $\Gamma$-convergence \cite{DG-S}. In his terminology, the $\Gamma$-limit in the $L^2$-topology is given by the functional
\[
\int_\Omega \langle A_{\rm hom}\nabla u, \nabla u\rangle dx.
\]
The symmetric {\em homogenized matrix} is characterized by a cell-problem formula
\[
\langle A_{\rm hom}\xi, \xi\rangle=\min\Big\{ \int_{(0,1)^d}a(y)|\xi+\nabla \phi|^2dy: \phi \ 1\hbox{-periodic}\Big\}.
\]
A simple but important observation is that $A_{\rm hom}$ depends on the ``geometry'' of $a$, and in particular it is not equal to the average of $a$: that is, to $\overline a=\int_{(0,1)^d}a(y)dy$,
except in the trivial case where $a$ is constant.
This result may also be viewed as a consequence of previous results on the $G$-convergence and $H$-convergence of elliptic differential operators \cite{Spa,Mu-Ta,Cio-Don}, and has been generalized in many ways (for a general reference from the $\Gamma$-convergence standpoint, see \cite{BDF}). 

Following this variational approach to homogenization, in the framework of fractional Sobolev spaces, we therefore consider quadratic energies with oscillating coefficients, analogous to the $H^1$-case, of the form 
\[
(1-s)\int_\Omega\int_\Omega a\left(\frac{x}\epsilon,\frac{y}\epsilon\right)
\frac{|u(x)-u(y)|^2}{|x-y|^{d+2s}}dxdy,
\]
where now the coefficient $a$ is $1$-periodic in both variables. Here we have two parameters $s$, tending to $1$, and $\epsilon$, tending to $0$. Formally, letting $\epsilon$ be fixed and $s$ tend to $1$, we can apply the Bourgain--Brezis--Mironescu result, and obtain a limit of the form
\[
\frac{\sigma_{d-1}}{2d}\int_\Omega a\left(\frac{x}\epsilon,\frac{x}\epsilon\right)|\nabla u|^2dx,
\]
where, due to the singularity of the kernel, only the values of $a$ on the diagonal $x=y$ contribute. We can then apply the classical theory of elliptic homogenization. The limit is given by $\frac{\sigma_{d-1}}{2d} F_{\rm hom}$ with
\begin{equation}\label{fom}
F_{\rm hom}(u)= 
\int_\Omega \langle A_{\rm hom}\nabla u, \nabla u\rangle dx,
\end{equation}
where now $A_{\rm hom}$ satisfies
\begin{equation}\label{fomi}
\langle A_{\rm hom}\xi, \xi\rangle=\min\Big\{ \int_{(0,1)^d}a(y,y)|\xi+\nabla \phi|^2dy: \phi \ 1\hbox{-periodic}\Big\}.
\end{equation}
This description continues to hold if $1-s$ tends to $0$ sufficiently fast compared with $\epsilon$, as follows from the general theory of $\Gamma$-convergence \cite{DM,B-LN14}. However, it cannot hold for an arbitrary convergence of $\epsilon \to 0$ and $s\to 1$. Indeed, letting $\epsilon\to 0$  with fixed $s$, the limit of these functionals is
\[
(1-s)\overline a\int_\Omega\int_\Omega 
\frac{|u(x)-u(y)|^2}{|x-y|^{d+2s}}dxdy,
\]
where \[\overline a=\int_{(0,1)^{d}}\int_{(0,1)^{d}}a(x,y)dxdy.
\] This can be proved using the compact embedding of $H^s$ into $L^2$, and the weak convergence of the coefficients to $\overline a$ (see e.g.~\cite{Focardi}). Hence, if $s$ tends to $1$ slow enough with respect to how $\epsilon$ tends to $0$, the $\Gamma$-limit of our functionals is simply $\frac{\sigma_{d-1}}{2d}\overline a $ times the Dirichlet integral. This shows that in general the behavior of $F_\epsilon$ must depend on the relation between the two infinitesimal quantities $\epsilon$ and $1-s$.

\smallskip
The main result of this paper is a complete analysis of the behavior as $s$ tends to $1$ and $\epsilon$ tends to $0$ simultaneously, given by the following theorem.

\begin{theorem}[A BBM homogenization theorem]
\label{main}
Let $a:\mathbb{R}^d\times\mathbb{R}^d\to \mathbb{R}$ be a continuous function, $1$-periodic in both variables and such that there exist $0<\alpha\leq \beta$ for which
$    \alpha\leq a\left(x,y\right)\leq \beta$
for all $x,y\in\mathbb{R}^d$. 
    Let $\Omega \subset \R^d$ be an open, bounded set with Lipschitz boundary. Let $\epsilon \in (0,1)$ and $s \coloneqq s(\epsilon) \in (0,1)$. Let 
\[
F_{\epsilon}(u):=(1-s)\int_\Omega\int_\Omega a\left(\frac{x}\epsilon,\frac{y}\epsilon\right)
\frac{|u(x)-u(y)|^2}{|x-y|^{d+2s}}dxdy
\]    
be defined in $H^s(\Omega)$.
    If $\lim\limits_{\epsilon \to 0} s = 1$ and
    \[
        \lim_{\epsilon \to 0} \epsilon^{2(1-s)} = \lambda \in [0,1], 
    \]
    then
    \[  
        \glim_{\epsilon \to 0} F_\epsilon(u) = \frac{\sigma_{d-1}}{2d}\left((1-\lambda)\overline{a}\|\nabla u\|_{L^2(\Omega)}^2+\lambda F_{\rm hom}(u)\right),
    \]
    where $u \in H^1(\Omega)$, the $\Gamma$-limit is calculated with respect to the $L^2(\Omega)$ strong topology and $\sigma_{d-1}$ is the $(d-1)$-Hausdorff measure of the unit sphere $\mathbb{S}^{d-1}$ in $\R^d$. 
\end{theorem}

We note that $\epsilon^{2(1-s)}\in (0,1)$ for all $\epsilon,s\in(0,1)$, so that the existence of a limit $\lambda\in[0,1]$ always holds, up to subsequences. Furthermore, since the functionals are comparable with the Bourgain--Brezis--Mironescu functionals, this theorem is complemented by a compactness result: if $F_\epsilon(u_\epsilon)\le S<+\infty$, then, up to subsequences and addition of constants, $u_\epsilon$ converges in $L^2(\Omega)$ to some $u\in H^1(\Omega)$ (see \cite{ponce}).

The interesting regime is at the {\em critical scaling} $1-s\sim\frac1{|\log \epsilon|}$, where homogenization and averaging effects coexist. This logarithmic scaling reflects the fact that the BBM energy distributes its mass equally over logarithmic scales.
This can be explained by looking at the recovery sequences for a linear function $\xi \cdot x$. These recovery sequences are still of the form 
$u_\epsilon(x)=\xi \cdot x+\epsilon \varphi(x/\epsilon)$ with $\varphi$ $1$-periodic. If $|x-y|<\!< \epsilon$ then
$$|u_\epsilon(x)-u_\epsilon(y)|^2
\sim |\langle\xi +\nabla \varphi(x),x-y\rangle|^2$$
and, after integrating  in $z=x-y$, one recovers the formula for $A_{\rm hom}$, while if $|x-y|>\!> \epsilon$ then $
\epsilon\varphi$ is negligible, and
$$|u_\epsilon(x)-u_\epsilon(y)
|^2
\sim |\langle\xi, x-y\rangle|^2,
$$
so that, after integrating in $z=x-y$, the coefficient $a$ is averaged. Interactions for which $|x-y|$ is of order $\epsilon$ are negligible. The outcome is a linear combination of the two extreme regimes through the coefficient $\lambda$. The coexistence of different behaviours in a homogenization problem can also be found in the analysis of Ginzburg--Landau vortices in heterogeneous media \cite{ABCDP},  for capacitary problems in a periodic environment \cite{BB}, or for fractional phase transitions \cite{BK}. All of these situations are characterized by the presence of logarithmic scalings. In those cases, however, the energy either concentrates at scales where the coefficient $a$ is constant, or where $a$ is oscillating. In the first case, $a$ is minimized and in the second case $a$ is either homogenized or averaged. 

\smallskip
Theorem \ref{main} also highlights two regimes in which we have {\em separation of scales}. In the
{\em subcritical regime} $\lambda = 0$;
that is, if $
1-s>\!>\frac1{|\log \epsilon|}$, the limit coincides with the iterated limit obtained by letting first $\epsilon\to0$ and then $s\to 1$.
In the {\em supercritical regime} $\lambda = 1$;
that is, if 
$1-s<\!<\frac1{|\log \epsilon|}$,
the limit coincides with the iterated limit obtained by letting first  $s\to 1$  and then $\epsilon\to0$.
The behaviour is in sharp contrast with convolution-type energies, where the critical scaling leads to a genuinely new homogenized cell problem rather than coexistence of two limiting mechanisms (see the next subsection).

\smallskip
The proof relies on splitting the energy into short- and long-range interactions. The short-range contribution is treated using the homogenization result obtained in \cite {BBD}, whereas the long-range contribution requires a blow-up argument showing that oscillations average out. The critical regime is characterized by the coexistence of these two mechanisms.
We note that in \cite {BBD} the supercritical separation of scales is shown to appear in the non-sharp regime $1-s<\!<\epsilon^2$. However, an inspection of the proof shows that it also holds for $1-s<\!<
|\log\epsilon|^{-1}$. The supercritical scaling also appears in the study of fractional thin films  of thickness $\epsilon$, which is the scaling corresponding to a complete separation of scales in the dimension-reduction process in the Bourgain--Brezis--Mironescu limit \cite{BPS}. Conversely, if $1-s>\!>
|\log\epsilon|^{-1}$, then interactions up to scale $
\varepsilon$ are negligible, which explains the behaviour in the subcritical regime. 

Throughout the paper we restrict our analysis to Gagliardo seminorms in the quadratic case $p=2$. This choice is not dictated by the nature of the phenomenon we describe --- the coexistence of homogenization and averaging effects at the critical scaling --- but simplifies the exposition. We expect the same qualitative picture, including the convex combination of homogenized and averaged energies at the critical scaling, to hold for general periodic energies. In place of the oscillating $2$-Gagliardo seminorms one can consider $p$-Gagliardo seminorm, or more general functionals that extend Gagliardo seminorms in the spirit of what has been done for convolution functionals in \cite[Chapter 6]{AABPT}, with the necessary changes to the blow-up and homogenization arguments.

\subsection*{Comparison with convolution functionals}\label{Compa}
In order to highlight the different nature of nonlocality, it is worth comparing the behaviour of fractional energies and that of  energies with {\em convolution} kernels. For such energies, the theorem of Bourgain--Brezis--Mironescu can still be applied. In the quadratic case, we consider functionals of the form
\[
\iint_{\Omega\times\Omega} \frac1{\epsilon^{d+2}}
\rho\left(\frac{x-y}\delta\right)\,
a\left(\frac{x}\epsilon,\frac{y}\epsilon\right)
\frac{|u(x)-u(y)|^2}{|x-y|^{d+2s}}dxdy.
\]
 For these energies, the {\em critical scaling} is $\epsilon\sim\delta$. In this regime, the scale of the kernel and the scale of the periodicity are the same. Their interaction produces a quadratic limit with homogenized matrix characterized by a {\em nonlocal cell-problem formula}. If $\epsilon/\delta
\to 1$, we have
\begin{eqnarray*}
&&\hskip-.7cm\langle A_{\rm hom}\xi, \xi\rangle=\min\Big\{ \iint_{(0,1)^d\times \mathbb R^d}
\rho(x-y) a(x,y)|\xi(x-y)+\varphi(x)-\varphi(y)|^2dxdy:\\&&
\hskip10cm\varphi \ 1\hbox{-periodic}\big\}
\end{eqnarray*}
(see \cite{AABPT}), while we have separation of scales in the other regimes \cite{Brusca}. Hence, contrary to the fractional case, for convolution energies, at the critical scaling we do not have coexistence of the two separation-of-scale regimes.

\subsection*{Plan of the paper}
In Section \ref{lemmata} we gather some preliminary results on the asymptotic behaviour of Gagliardo seminorms: on the one hand, results that allow us to reduce the range of interactions without changing the limit, and, on the other hand, a Riemann--Lebesgue-type lemma, which describes the behaviour of long-range interactions for Gagliardo-type integrals with periodic coefficients. Section \ref{sec:liminf} is dedicated to the proof of the liminf inequality for Theorem \ref{main}. The main argument is to subdivide the computation into short-range and long-range interactions. The short-range interactions are homogenized using the result in \cite{BBD}, which is valid in our setting. For the long-range interactions, we use a blow-up and periodization argument that allows us to consider only $\epsilon$-periodic perturbations of affine functions. Since these perturbations are negligible at long-range, we can consider just linear functions, and compute the corresponding limit using the Riemann--Lebesgue type lemma from Section~\ref{lemmata}. Finally, we complete the proof of Theorem \ref{main} by showing that test functions for the classical homogenization formula provide recovery sequences for affine functions also in the nonlocal case. Again, we show that the relevant energy is due only to short-range and long-range contributions, while interactions in the ``intermediate range" are negligible, due to logarithmic scaling of the energies. 

\section{Notation and preliminary results}\label{lemmata}
We use standard notation for Sobolev and Lebesgue spaces. The characteristic function of a set $A$ will be denoted by $\1_A$; if $A\in\mathbb{R}^d$ is Borel, then both $\mathscr{L}^d(A)$ and $|A|$ denote its Lebesgue measure. The notation $\mu_j \rightharpoonup\mu$ is used for the narrow convergence of measures (in duality with $\mathcal{C}_b$).

\subsection{Asymptotics of double integrals}\label{lemmata1}
In this section, we include some results regarding the behaviour of double integrals. The first is a general estimate in \cite[Theorem 1]{BBM}.

\begin{theorem}\label{BBM_theorem1}
    Let $\Omega$ be a bounded connected Lipschitz open set. Let $f\in H^1(\Omega)$ and $\rho\in L^1(\R^d)$, with $\rho\geq 0.$ Then there exists $C=C(\Omega)>0$ such that
\begin{equation}\label{BBM_formula}
    \iint_{\Omega^2}\dfrac{|f(x)-f(y)|^2}{|x-y|^2} \rho(x-y) dx dy \leq C \|\nabla f\|_{L^2(\Omega)}^2\|\rho\|_{L^1(\R^d)}.
\end{equation}
\end{theorem}

\bigskip The following results treat more specifically Gagliardo seminorms.
For future reference, it is convenient to introduce a positive parameter $\epsilon$, and consider families $s=s(\epsilon)$ satisfying $s\to 1^-$ as $\epsilon\to 0$. Note that in the statements we do not specify the dependence of the parameter $s$ on $\epsilon$. To simplify the notation, we will adopt this convention in the rest of the paper for parameters dependent on $\epsilon$.

\smallskip
The first result is a quantitative version of \cite[Lemma 3]{BBD}; we include the proof for the sake of completeness.

\begin{lemma}\label{l_2_large_interactions}
    There exists $C=C(d)>0$ such that
    \begin{equation}\label{large_interactions_estimate}
        (1-s)\iint_{A^2\cap\{|x-y|\geq r\}} \dfrac{|f(x)-f(y)|^2}{|x-y|^{d+2s}}dxdy\leq \dfrac{C(1-s)}{r^{2s}}\|f\|^2_{L^2(A)}
    \end{equation}
     for every bounded measurable set $A\subset\R^d$, every $r>0$ and every $f\in L^2(A)$.
    In particular,
    if $(u_\epsilon)\subset L^2(\Omega)$ is a bounded sequence and $r\coloneqq r(\epsilon)>0$ satisfies
    \begin{equation}
        \lim\limits_{\epsilon\to 0}\frac{1-s}{r^2}=0.
    \end{equation}
    Then,
    \begin{equation}
        \lim\limits_{\epsilon\to 0}(1-s)\iint_{\Omega^2\cap\{|x-y|\geq r\}}\dfrac{|u_\epsilon(x)-u_\epsilon(y)|^2}{|x-y|^{d+2s}}dxdy=0.
    \end{equation}
\end{lemma}
\begin{proof}
    Using $|f(x)-f(y)|^2\leq 2|f(x)|^2+2|f(y)|^2$, and integrating in polar coordinates, we obtain
        $$(1-s)\iint_{A^2\cap\{|x-y|\geq r\}} \dfrac{|f(x)-f(y)|^2}{|x-y|^{d+2s}}dxdy
        \leq 4(1-s)\|f\|^2_{L^2(A)}\sigma_{d-1}\int_r^{+\infty}t^{-1-2s}dt,
    $$
    which gives \eqref{large_interactions_estimate}. Applying \eqref{large_interactions_estimate} with $A=\Omega$ and $f=u_\epsilon$, and using $r^{-2s}\leq\max\{r^{-2},1\}$, we obtain
        $$(1-s)\iint_{\Omega^2\cap\{|x-y|\geq r\}} \dfrac{|u_\epsilon(x)-u_\epsilon(y)|^2}{|x-y|^{d+2s}}dxdy
        \leq C\max\Big\{\frac{1-s}{r^2},1-s\Big\}\|u_\epsilon\|^2_{L^2(\Omega)},
  $$ 
    which vanishes as $\epsilon\to0$, since $(u_\epsilon)$ is bounded in $L^2(\Omega)$.
\end{proof}

The following lemma states that interactions at order $\epsilon$ are asymptotically negligible in these integrals, provided that the sequences satisfy an $\epsilon$-bound.

\begin{lemma}\label{l_infty_large_interactions}
    Let $(v_\epsilon)\subset L^{\infty}(\Omega)$ and assume that there exists $C>0$ such that $\|v_\epsilon\|_{\infty}\leq C\epsilon$. Then, for every $M>0$, we have
    \begin{equation}
        \lim\limits_{\epsilon\to 0}(1-s)\iint_{\Omega^2\cap\{|x-y|\geq M\epsilon\}} \dfrac{|v_\epsilon(x)-v_\epsilon(y)|^2}{|x-y|^{d+2s}}dxdy=0.
    \end{equation}
\end{lemma}
\begin{proof} Since $\Omega$ is bounded, there exists $R>0$ such that $x,y\in\Omega$ implies $|x-y|<R$.
Using the hypothesis on $(v_\epsilon)$, we get 
\begin{align}\label{lem_stima_1}
    \iint_{\Omega^2\cap\{|x-y|\geq M\epsilon\}} \dfrac{|v_\epsilon(x)-v_\epsilon(y)|^2}{|x-y|^{d+2s}}dxdy\leq 4C^2\epsilon^2\iint_{\Omega^2\cap\{R>|x-y|\geq M\epsilon\}}\dfrac{1}{|x-y|^{d+2s}}dxdy
\end{align}    
Integrating in polar coordinates, we may estimate \eqref{lem_stima_1} from above with
\begin{eqnarray*}  4C^2\epsilon^2\int_\Omega\int_{\mathbb{S}^{d-1}}\int_{M\epsilon}^{R}t^{-1-2s}dt d\sigma(\nu)dx
    =4C^2\epsilon^2|\Omega|\sigma_{d-1}\dfrac{(M\epsilon)^{-2s}-R^{-2s}}{2s}
    \leq D\epsilon^{2(1-s)}-E\epsilon^2,
\end{eqnarray*}
and $D, E>0$ are some fixed positive constants.
The claim follows by multiplying by $1-s$ and letting $\epsilon\to 0$.
\end{proof}

The following lemma states that double integrals depending on sequences of equi-Lipschitz functions are asymptotically negligible on sets with vanishing measure.

\begin{lemma}\label{w_1_infty_close_interactions}
    Let $(\Omega_\epsilon)$ be a sequence of measurable subsets of $\Omega$ and let $(f_\epsilon)$ be a sequence of uniformly Lipschitz functions on $\Omega$. Assume that
        $\lim\limits_{\epsilon\to 0}|\Omega_\epsilon|=0$.
    Then, it follows that  
    \begin{equation}
        \lim\limits_{\epsilon\to 0}\ (1-s)\int_{\Omega_\epsilon} \int_{\Omega} \dfrac{|f_\epsilon(x)-f_\epsilon(y)|^2}{|x-y|^{d+2s}}dxdy =0
    \end{equation}
\end{lemma}
\begin{proof} 
By hypothesis, there exists $L>0$ such that
$|f_\epsilon(x)-f_\epsilon(y)| \leq L|x-y|$
for every $x,y \in \Omega$ and for every $\epsilon > 0$. 
We then obtain
\begin{align}
 (1-s)\int_{\Omega_\epsilon} \int_{\Omega} \dfrac{|f_\epsilon(x)-f_\epsilon(y)|^2}{|x-y|^{d+2s}}dxdy
    \label{lem_stima_3} \leq 2L^2  (1-s) \int_{\Omega_\epsilon} \int_{\Omega}\dfrac{1}{|x-y|^{d+2s-2}}dxdy
\end{align}
We can estimate \eqref{lem_stima_3} from above, integrating in polar coordinates and using that $\Omega \subset B_R(0)$ for some $R > 0$, obtaining
\begin{eqnarray*}
    (1-s)\int_{\Omega_\epsilon} \int_{\Omega}\dfrac{1}{|x-y|^{d+2s-2}}dxdy
    &\leq& (1-s)\int_{\Omega_\epsilon}\int_{\mathbb{S}^{d-1}}\int_0^{2R}t^{1-2s}dt d\sigma(\nu)dx\\
    &=&\frac {|\Omega_\epsilon|\,\sigma_{d-1}\,(2R)^{2(1-s)}}2,
\end{eqnarray*}
which tends to $0$, since $(2R)^{2(1-s)}$ is bounded and $|\Omega_\epsilon|\to 0$.
\end{proof}

\subsection{Averaged limits of Gagliardo seminorms}
The next lemma is a sort of quantified Riemann--Lebesgue on the average behaviour of double integrals, in which we take into account periodic coefficients. In order to state this result, we need some specific notation.

Consider a continuous function $a:\mathbb{R}^d\times\mathbb{R}^d\to \mathbb{R}$, which is $1$-periodic in both variables. We set $\beta=\max |a|$, and 
\begin{equation}\label{mean_a}
    \overline a\coloneqq\iint_{(0,1)^d\times(0,1)^d} a (x,y )dxdy,
\end{equation}
the average of $a$.

For $\epsilon>0$ and $k\in\Z^d$, we denote by $Q_{\epsilon,k}\coloneqq \epsilon\left(k+[0,1)^d\right)$, the cubes of the partition of $\R^d$ induced by the lattice $\epsilon\Z^d$ and, for measurable sets $A\subset\R^d$, we set
\[
    \partial_\epsilon A\coloneqq
    \big\{x\in\R^d\ | \operatorname{dist}(x,\partial A)\leq \sqrt d\,\epsilon\big\}.
\]

\begin{lemma}[Averaging interactions at large scales]\label{l_averaging_large_scale}
    There exists $C=C(d)>0$ such that, for every $s\in[3/4,1)$, $\epsilon\in(0,1)$, $\epsilon\leq r_1<r_2\leq 1$, $\xi\in\R^d$ and every bounded measurable set $A\subset\R^d$, it holds
    \begin{align}\label{averaging_estimate}
        \nonumber&\bigg\vert(1-s)\int_{A}\int_{\{r_1\leq|x-y|\leq r_2\}} a\left(\frac x\epsilon,\frac y\epsilon\right)\dfrac{|\langle\xi,x-y\rangle|^2}{|x-y|^{d+2s}}dydx-\overline{a}\,\dfrac{\sigma_{d-1}}{2d}|\xi|^2\Big(r_2^{2(1-s)}-r_1^{2(1-s)}\Big)|A|\bigg\vert\\
        &\leq C\beta|\xi|^2\Big((1-s)\dfrac{\epsilon}{r_1}|A|+|\partial_\epsilon A|\Big).
    \end{align}
\end{lemma}
\begin{proof}
In the following, $C=C(d)>0$ may change from line to line. Set
\[
    \kappa(x,y)\coloneqq(1-s)\dfrac{|\langle\xi,x-y\rangle|^2}{|x-y|^{d+2s}},\qquad \Lambda\coloneqq\big(A\times\R^d\big)\cap\{r_1\leq|x-y|\leq r_2\}.
\]
Integrating in polar coordinates and using that, by symmetry,
\begin{equation}\label{sphere_average}
    \int_{\mathbb{S}^{d-1}}|\langle \xi,\nu\rangle|^2d\sigma(\nu)=\dfrac{\sigma_{d-1}}{d}|\xi|^2
\end{equation}
for all $\xi\in\R^d$, we get
\begin{equation}\label{kappa_mass}
    \int_{\Lambda}\kappa(x,y)dydx=\dfrac{\sigma_{d-1}}{2d}|\xi|^2\big(r_2^{2(1-s)}-r_1^{2(1-s)}\big)|A|,
\end{equation}
so that the left-hand side of \eqref{averaging_estimate} equals the absolute value of
\[
    \int_\Lambda\left(a\left(\frac x\epsilon,\frac y\epsilon\right)-\overline{a}\right)\kappa(x,y)dydx.
\]
We partition $\Lambda$ into the sets
\begin{align}
    \Lambda_{1,\epsilon}&\coloneqq\bigcup\left\{Q_{\epsilon,j}\times Q_{\epsilon,k}\,\middle\vert\, Q_{\epsilon,j}\times Q_{\epsilon,k}\subset \Lambda,\ Q_{\epsilon,j}\subset A\right\},\\
    \Lambda_{2,\epsilon}&\coloneqq \Lambda\cap\left(\partial_\epsilon A\times\R^d\right),\\
    \Lambda_{3,\epsilon}&\coloneqq \Lambda\setminus\left(\Lambda_{1,\epsilon}\cup\Lambda_{2,\epsilon}\right).
\end{align}
Note that $\Lambda_{3,\epsilon}$ is covered by the cubes $Q_{\epsilon,j}\times Q_{\epsilon,k}$ with $Q_{\epsilon,j}\subset A$ that meet $\Lambda$ without being contained in it.

Let $Q_{\epsilon,j}\times Q_{\epsilon,k}\subset\Lambda_{1,\epsilon}$ and fix $(\overline x,\overline y)\in Q_{\epsilon,j}\times Q_{\epsilon,k}$. Since $a$ is $1$-periodic in both variables, we have
\begin{equation}\label{cell_average_a}
    \int_{Q_{\epsilon,j}}\int_{Q_{\epsilon,k}} a\left(\frac x\epsilon,\frac y\epsilon\right)dydx=\epsilon^{2d}\,\overline{a}
\end{equation}
 for all $j,k\in\Z^d$, so that
\begin{align}
    \nonumber\int_{Q_{\epsilon,j}}&\int_{Q_{\epsilon,k}}\left(a\left(\frac x\epsilon,\frac y\epsilon\right)-\overline{a}\right)\kappa(x,y)dydx\\
    \label{cell_error}&=\int_{Q_{\epsilon,j}}\int_{Q_{\epsilon,k}}\left(a\left(\frac x\epsilon,\frac y\epsilon\right)-\overline{a}\right)\left(\kappa(x,y)-\kappa(\overline x,\overline y)\right)dydx.
\end{align}
In $\{|x-y|\geq r_1\}$, we have
\begin{equation}\label{kernel_gradient}
    |\nabla_{(x,y)}\kappa(x,y)|\leq \sqrt 2(d+4)(1-s)|\xi|^2|x-y|^{1-d-2s},
\end{equation}
and two points $(x,y)$, $(x',y')$ of $Q_{\epsilon,j}\times Q_{\epsilon,k}$ satisfy $|x'-y'|\geq r_1\geq\epsilon$ and $|x-y|\leq|x'-y'|+2\sqrt d\,\epsilon\leq(1+2\sqrt d)|x'-y'|$. Hence, by the mean-value theorem,
\begin{equation}\label{kernel_oscillation}
    |\kappa(x,y)-\kappa(\overline x,\overline y)|\leq C\epsilon(1-s)|\xi|^2|x-y|^{1-d-2s}
\end{equation}
for $(x,y)\in Q_{\epsilon,j}\times Q_{\epsilon,k}$.

Now, we sum \eqref{cell_error} over the cubes of $\Lambda_{1,\epsilon}$ and use \eqref{kernel_oscillation}, $|a|\le \beta$, $2s-1\geq1/2$, and $r_1^{2(1-s)}\leq1$. Thus, we obtain
\begin{align}
    \nonumber\Big|\int_{\Lambda_{1,\epsilon}}&\left(a\left(\frac x\epsilon,\frac y\epsilon\right)-\overline{a}\right)\kappa(x,y)dydx\Big|\\
    \nonumber&\leq\sum_{Q_{\epsilon,j}\times Q_{\epsilon,k}\subset\Lambda_{1,\epsilon}}\int_{Q_{\epsilon,j}}\int_{Q_{\epsilon,k}}\left\vert a\left(\frac x\epsilon,\frac y\epsilon\right)-\overline{a}\right\vert\left\vert\kappa(x,y)-\kappa(\overline x,\overline y)\right\vert dydx\\
    \nonumber&\leq 2\beta C\epsilon(1-s)|\xi|^2\int_A\int_{\{r_1\leq|x-y|\leq r_2\}}\dfrac{1}{|x-y|^{d+2s-1}}dydx\\
    \nonumber&= 2\beta C\epsilon(1-s)|\xi|^2|A|\,\sigma_{d-1}\int_{r_1}^{r_2}t^{-2s}dt\\
    &\leq C\beta\sigma_{d-1}(1-s)|\xi|^2\dfrac{\epsilon}{r_1}|A|.
\end{align}
On $\Lambda_{2,\epsilon}$ we use $|a|\le \beta$ and \eqref{kappa_mass}, which give the bound $\beta\frac{\sigma_{d-1}}{2d}|\xi|^2|\partial_\epsilon A|$. Finally, if $Q_{\epsilon,j}\times Q_{\epsilon,k}$ meets $\Lambda$ without being contained in it, then it contains a pair $(x',y')$ with $|x'-y'|\in\{r_1,r_2\}$, hence all its points satisfy $$|x-y|\in[r_1-2\sqrt d\,\epsilon,r_1+2\sqrt d\,\epsilon]\cup[r_2-2\sqrt d\,\epsilon,r_2+2\sqrt d\,\epsilon].$$ Therefore,
\begin{align}
\nonumber\Lambda_{3,\epsilon}\subset(A\times\R^d)\cap\big(\{r_1\leq|x-y|\leq r_1+2\sqrt d\,\epsilon\}\cup\{r_2-2\sqrt d\,\epsilon\leq|x-y|\leq r_2\}\big),
\end{align}
and, since $t\mapsto t^{1-2s}$ is decreasing, the integral of $\kappa$ on each of the two sets on the right-hand side is at most $2\sqrt d\,\sigma_{d-1}(1-s)|\xi|^2\epsilon\, r_1^{1-2s}|A|$. Using $|a|\le \beta$ once more, the contribution of $\Lambda_{3,\epsilon}$ is bounded by $C\beta\sigma_{d-1}(1-s)|\xi|^2\frac{\epsilon}{r_1}|A|$, and \eqref{averaging_estimate} follows.
\end{proof}

\section{Liminf inequality}\label{sec:liminf}
We now prove Theorem \ref{main}. We prove the lower bound in this section and the upper bound in the next. It is not restrictive to suppose that $\Omega$ is connected up to restricting to interactions in each connected component. This will allow to use Theorem \ref{BBM_theorem1}.

We recall that $a:\mathbb{R}^d\times\mathbb{R}^d\to \mathbb{R}$ is a continuous function, $1$-periodic in both variables and such that there exist $0<\alpha\leq \beta$ for which
\begin{equation}\label{coercivity_boundedness_a}
    \alpha\leq a\left(x,y\right)\leq \beta
\end{equation}
for all $x,y\in\mathbb{R}^d$. Its average value is denoted by $\overline a$. The coefficient $\lambda $ is defined as the limit of $\epsilon^{2(1-s)}$. We study the asymptotic behaviour of 
$$
    F_\epsilon(u)= (1-s)\iint_{
\Omega\times\Omega} a\left(\frac{x}\epsilon,\frac{y}\epsilon\right)
\frac{|u(x)-u(y)|^2}{|x-y|^{d+2s}}dxdy.
$$
Note that by \eqref{coercivity_boundedness_a} the sequence $F_\epsilon$ inherits the coerciveness properties of Gagliardo seminorms, and, as a consequence, the domain of the $\Gamma$-limit functional is $H^1(\Omega)$ \cite{ponce}. Moreover, by \eqref{coercivity_boundedness_a}, we may apply to these functionals the results in Section~\ref{lemmata1}.

\begin{remark}\rm A minor technical note concerns the role of the continuity of $a$, which is only necessary to use the diagonal function $a(x,x)$ when considering short-range interactions. The condition is not necessary, for example, if $a=a(x)$ or if $\lambda =0$.
\end{remark}

To prove the $\Gamma$-liminf inequality, we split the functional as a sum of its short- and long-range interactions. 

Given a measurable set $\Lambda\subset\Omega^2$, consider the localized functional
\begin{equation}
\label{truncatedfunctional}
    F_\epsilon(u,\Lambda)\coloneqq (1-s)\iint_\Lambda a\left(\frac{x}\epsilon,\frac{y}\epsilon\right)
\frac{|u(x)-u(y)|^2}{|x-y|^{d+2s}}dxdy.
\end{equation}
Fix $M > 0$ and choose a parameter $r (\coloneqq r(\epsilon))$ such that
\begin{equation}\label{choice_auxiliary_microscopic_parameter}
    r > 0, \qquad \lim_{\epsilon \to 0} \frac{r}\epsilon = 0, \qquad \lim_{\epsilon \to 0} r^{2(1-s)} = \lambda.
\end{equation}
For each $\epsilon > 0$, we divide $\Omega^2$ into the following sets
\begin{align}
    \label{liminf_set_A1} A_{1,\epsilon}&\coloneqq\left\{(x,y)\in\Omega^2\mid |x-y|\geq M\epsilon \right\}, \\
    \label{liminf_set_A2} A_{2,\epsilon}&\coloneqq\left\{(x,y)\in\Omega^2\mid M\epsilon > |x-y| > r \right\}, \\
    \label{liminf_set_A3} A_{3,\epsilon}&\coloneqq\left\{(x,y)\in\Omega^2\mid r \geq |x-y| \right\}.
\end{align}
We note that for every $\epsilon>0$ and every $v\in L^2(\Omega)$, it holds
\begin{equation}
    F_\epsilon(v)=\sum_{k=1}^3 F_{\epsilon}(v,A_{k,\epsilon}).
\end{equation}

Now, consider $u \in H^1(\Omega)$ and a sequence $(u_\epsilon) \subseteq L^2(\Omega)$ such that $u_\epsilon \to u$ in $L^2(\Omega)$. Without loss of generality, we can assume that there exists $C > 0$ such that
\begin{equation}
\label{equicoercivity}
    F_\epsilon(u_\epsilon) \leq C.
\end{equation}

Throwing away the contribution from $A_{2,\epsilon}$, we obtain
\begin{equation}\label{sudd_liminf}    \liminf_{\epsilon \to 0} F_\epsilon(u_\epsilon) \geq \liminf_{\epsilon \to 0}F_\epsilon(u_\epsilon, A_{1,\epsilon}) + \liminf_{\epsilon \to 0}F_\epsilon(u_\epsilon, A_{3,\epsilon}).
\end{equation}
We now estimate each term separately. For the second term we exploit the $\Gamma$-liminf inequality obtained by Braides, Brusca, and  Donati in \cite{BBD}. 

\begin{lemma}[Homogenization of interactions at small scales]\label{l_short_range}
    Let $u \in H^1(\Omega)$. Consider a sequence $(u_\epsilon) \subseteq L^2(\Omega)$ such that $u_\epsilon \to u$ in $L^2(\Omega)$. Assume that there exists $C > 0$ such that 
    \[
        F_\epsilon(u_\epsilon, A_{3,\epsilon}) \leq C
    \]
    for every $\epsilon > 0$. Then
    \begin{equation}
        \liminf_{\epsilon \to 0} F_\epsilon(u_\epsilon, A_{3,\epsilon}) \geq \frac{\sigma_{d-1}}{2d} \lambda F_{\rm hom}(u),
    \end{equation}
      where $F_{\rm hom}$ is given by \eqref{fom} and \eqref{fomi}.
    \end{lemma}

\begin{proof} We can follow the proof of the $\Gamma$-liminf inequality in \cite{BBD} essentially verbatim. 
Their argument (Section 3.1) is that, given a sequence $u_\epsilon$ convergent to some $u$ in $L^2(\Omega)$, 
by discretization and averaging of the functions $u_\epsilon$, one can construct a sequence $\overline u_\epsilon$ such that 
$$
F_\epsilon(u_\epsilon,A_{3,\epsilon})\ge r^{2(1-s)}\frac{\sigma_{d-1}}{2d} \int_{\Omega'}a\Big(\frac{x}\epsilon,\frac{x}\epsilon\Big) |\nabla \overline u_\epsilon|^2dx
$$
for all $\Omega'$ open set compactly contained in $\Omega$ and $
\epsilon $ small enough (see the last lines at page 12 in \cite{BBD} and \cite[Section 1.2.2]{Bruscatesi}). In their hypothesis $\lambda=1$ and  they show that this inequality actually holds asymptotically for the whole functional. This is not relevant for our case, since we directly estimate only short-range interactions. Thus, we apply this inequality together with the liminf inequality for the local homogenization theorem for the right-hand side. By the arbitrariness of $\Omega'$ and using \eqref{choice_auxiliary_microscopic_parameter}, the claim follows with the desired multiplicative constant $\lambda$.
\end{proof}

Now, for the first term in \eqref{sudd_liminf}, we employ a Fonseca--M\"uller-type blow-up argument
(see \cite{FM,BMS}). For each $\epsilon > 0$, we introduce the measures $\theta_\epsilon$ and $\mu_\epsilon$ on $\R^d$, absolutely continuous with respect to the restricted Lebesgue measure $\mathscr{L}^d\LLL\Omega$, defined by
\begin{align}
    \frac{d\theta_\epsilon}{d\mathscr{L}^d}(x) &\coloneqq (1-s) \int_\Omega a\left(\frac{x}\epsilon,\frac{y}\epsilon\right)
\frac{|u_\epsilon(x)-u_\epsilon(y)|^2}{|x-y|^{d+2s}}dy, \\
\frac{d\mu_\epsilon}{d\mathscr{L}^d}(x) &\coloneqq (1-s) \int_{\Omega \cap \{|x-y|\geq M\epsilon\}}a\left(\frac{x}\epsilon,\frac{y}\epsilon\right)
\frac{|u_\epsilon(x)-u_\epsilon(y)|^2}{|x-y|^{d+2s}}dy,
\end{align}
where $x \in \Omega$.
Note that
\[
    F_\epsilon(u_\epsilon) = \theta_\epsilon(\Omega), \qquad F_\epsilon(u_\epsilon,A_{1,\epsilon}) = \mu_\epsilon(\Omega).
\]
Up to subsequences, from now on we assume that the limit
    $\lim_{\epsilon \to 0} F_\epsilon(u_\epsilon) \in \R$ exists.
By the assumption of equicoercivity \eqref{equicoercivity}, the sequences $(\theta_\epsilon)$ and $(\mu_\epsilon)$ are equibounded, so up to subsequences we have 
$\theta_\epsilon \rightharpoonup \theta$ and $ \mu_\epsilon \rightharpoonup \mu$.
 In particular, since $\Omega$ is open, we have 
\[
    \liminf_{\epsilon \to 0} \mu_\epsilon(\Omega) \geq \mu(\Omega) = \int_\Omega \frac{d\mu}{d\mathscr{L}^d}(x) dx.
\]
Therefore, computing the $\Gamma$-liminf inequality for long-range interactions amounts to produce a lower bound for the density of $\mu$. By the Lebesgue differentiation theorem, for $\mathscr{L}^d$-almost every $x_0\in\Omega$, one has
\begin{equation}
    \frac{d\mu}{d\mathscr{L}^d}(x_0)=\lim\limits_{\rho\to 0}\frac{\mu(Q_\rho(x_0))}{\mathscr{L}^d(Q_\rho(x_0))},
\end{equation}
where $Q_\rho(x)$ denotes a closed hypercube of size $\rho$, centered at $x$. Since $Q_\rho(x_0)$ is closed, the narrow convergence $\mu_\epsilon\rightharpoonup\mu$ gives $\limsup_{\epsilon\to0}\mu_\epsilon(Q_\rho(x_0))\leq\mu(Q_\rho(x_0))$, so that 
\begin{equation}
    \frac{d\mu}{d\mathscr{L}^d}(x_0)\geq\lim\limits_{\rho\to 0}\limsup\limits_{\epsilon\to 0}\frac{\mu_\epsilon(Q_\rho(x_0))}{\mathscr{L}^d(Q_\rho(x_0))},
\end{equation}
By a diagonal argument, we choose subsequences (not labelled) $\epsilon\to 0$ and $\rho\,(=\rho(\epsilon))$ such that 
\begin{equation}
    \frac{d\mu}{d\mathscr{L}^d}(x_0)\geq\lim\limits_{\epsilon\to 0}\frac{\mu_\epsilon(Q_\rho(x_0))}{\mathscr{L}^d(Q_\rho(x_0))}.
\end{equation}
Moreover, noting that $\rho$ is free to tend to $0$ as slow as desired, we assume that
\begin{equation}\label{rho_asymptotics}
    \lim\limits_{\epsilon\to 0} \rho^{2(1-s)}=1,\qquad \lim\limits_{\epsilon\to 0}\frac\rho\epsilon=+\infty,\qquad \frac\rho\epsilon\in\N.
\end{equation}
Since the same argument applies to $\theta$, we may also assume that $x_0$ is a point of finite density for $\theta$, so that
\begin{equation}\label{theta_bound}
    \sup_{\epsilon>0}\frac{\theta_\epsilon(Q_\rho(x_0))}{\mathscr{L}^d(Q_\rho(x_0))}<+\infty.
\end{equation}

On the other hand, since $u\in H^1(\Omega)$, $\mathscr{L}^d$-almost every $x_0\in\Omega$ is a Lebesgue point for $u$, for $\nabla u$ and $u$ is differentiable at $x_0$ in the $L^2$ sense, namely, setting $z\coloneqq\nabla u(x_0)$ and
\begin{equation}\label{affine_blowup}
    \ell(x)\coloneqq u(x_0)+\langle z,x-x_0\rangle,
\end{equation}
we have
\begin{equation}
    \Big(\frac{1}{\rho^d}\int_{Q_\rho(x_0)}|u(x)-\ell(x)|^2dx\Big)^{1/2}=\,o(\rho).
\end{equation}
Again, by choosing $\rho\,(=\rho(\epsilon))$ tending to $0$ sufficiently slowly, using that $u_\epsilon\to u$ in $L^2$, we assume additionally that (up to subsequences)

\begin{equation}\label{blowup_convergence}
   \lim\limits_{\epsilon\to 0}\frac{1}{\rho^{d+2}}\int_{Q_\rho(x_0)}|u_\epsilon(x)-\ell(x)|^2dx=0 .
\end{equation}
Indeed, $\rho^{-(d+2)}\int_{Q_\rho(x_0)}|u-\ell|^2dx=o(1)$ by the $L^2$-differentiability of $u$ at $x_0$, while $\rho^{-(d+2)}\|u_\epsilon-u\|^2_{L^2(\Omega)}\to0$ as soon as $\rho$ goes to $0$ sufficiently slowly.

Now, in the blow-up argument, we change the boundary data to match the affine limit. Fix an integer $N\geq 2$ and, for each $\epsilon$ and each $i=1,\ldots,N$, define 
\begin{equation}
    Q_{\rho,i}\coloneqq\{x\in Q_\rho(x_0)\mid \mathrm{dist}(x,Q_{\rho}(x_0)^c)> i w\},\qquad w\coloneqq\frac\rho{N^2}.
\end{equation}
In addition, for each $\epsilon$ and each $i=1,\ldots,N$, choose a smooth function $\psi_{\epsilon,i}\in\mathcal{C}^{\infty}_c(\mathbb{R}^d)$, such that
\begin{equation}
    \label{properties_of_the_cutoff_liminf}
    \1_{Q_{\rho,i}} \leq \psi_{\epsilon,i}\leq \1_{Q_{\rho,i-1}}, \qquad \|\nabla \psi_{\epsilon,i}\|_\infty \leq 2/w,
\end{equation}
and define 
\begin{equation}\label{definition_v}
    v_{\epsilon,i}(x)\coloneqq u_\epsilon(x)\psi_{\epsilon,i}(x)+(1-\psi_{\epsilon,i}(x))\ell(x),\qquad x\in Q_\rho(x_0).
\end{equation}
Note that $v_{\epsilon,i}=\ell$ on the set $\{x\in Q_\rho(x_0)\mid\mathrm{dist}(x,Q_\rho(x_0)^c)\leq \rho/N\}$, whose contribution is of order $1/N$.
Note that, for each $i$ fixed, it still holds 
\begin{equation}
    \lim\limits_{\epsilon\to 0}\frac{1}{\rho^{d+2}}\int_{Q_\rho(x_0)}|v_{\epsilon,i}(x)-\ell(x)|^2dx=0,
\end{equation}
since $v_{\epsilon,i}-\ell=\psi_{\epsilon,i}(u_\epsilon-\ell)$ and $0\leq\psi_{\epsilon,i}\leq1$.

We are ready to choose the desired boundary value, using a classical argument by De Giorgi, already exploited in \cite{AABPT} for convolution-type double integrals.

\begin{lemma}[Change of boundary data]\label{l_change_of_boundary_data}
    Fix $x_0$ and $N$, and let $\epsilon\to 0$ and $\rho\,(=\rho(\epsilon))$ be as above. 
    Let $v_{\epsilon,i_\epsilon}$ be defined in \eqref{definition_v}. Then, there exist $C>0$, independent of $N$, and a sequence of indices $(i_\epsilon)\subset\{1,\ldots,N\}$ such that
    \begin{equation}\label{affine_datum_estimate}
        \liminf\limits_{\epsilon\to 0}\frac{\mu_\epsilon(Q_\rho(x_0))}{\mathscr{L}^d(Q_\rho(x_0))}\geq \liminf\limits_{\epsilon\to 0}\frac{F_\epsilon\left(v_{\epsilon,i_\epsilon},A_{1,\epsilon}\cap Q_\rho(x_0)^2\right)}{\rho^d}-\frac CN,
    \end{equation}
     The same holds with $\theta_\epsilon$ and $Q_\rho(x_0)^2$ in place of $\mu_\epsilon$ and $A_{1,\epsilon}\cap Q_\rho(x_0)^2$, for the same sequence $(i_\epsilon)$.
\end{lemma}
\begin{proof}
    In the following, $C>0$ may change from line to line and never depends on $N$. We also drop the index $\epsilon$ from $v_{\epsilon,i}$ and $\psi_{\epsilon,i}$, dropping the base point $x_0$ and writing $Q_\rho\coloneqq Q_\rho(x_0)$, $\Lambda_\epsilon\coloneqq A_{1,\epsilon}\cap Q_\rho^2$ and $\varphi_\epsilon\coloneqq u_\epsilon-\ell$. Since $F_\epsilon$ is invariant under addition of constants, we may also assume that $u(x_0)=0$, so that $\ell(x)=\langle z,x-x_0\rangle$ and $\|\ell\|^2_{L^2(Q_\rho)}\leq d|z|^2\rho^{d+2}$. By \eqref{theta_bound}, we have
    \begin{equation}\label{blowup_energy_bound}
        C_0\coloneqq\sup_{\epsilon>0}\frac{\theta_\epsilon(Q_\rho)}{\rho^d}<+\infty,\qquad \frac{F_\epsilon(u_\epsilon,\Lambda_\epsilon)}{\rho^d}\leq\frac{\mu_\epsilon(Q_\rho)}{\rho^d}\leq \frac{\theta_\epsilon(Q_\rho)}{\rho^d},
    \end{equation}

    First, we record two estimates on the affine function $\ell$. Integrating in polar coordinates and using \eqref{sphere_average}, for every measurable $E\subset Q_\rho$ we have
    \begin{equation}\label{affine_energy}
        \frac{F_\epsilon(\ell,E\times Q_\rho)}{\rho^d}\leq \beta|z|^2\frac{1-s}{\rho^d}|E|\sigma_{d-1}\int_0^{\sqrt d\rho}t^{1-2s}dt\leq \frac{\beta\sigma_{d-1}d}{2}|z|^2\frac{|E|}{\rho^d},
    \end{equation}
    where we also used that $\rho^{2(1-s)}\leq 1$. Since $|\{x\in Q_\rho\mid\mathrm{dist}(x,Q_\rho^c)\leq\rho/N\}|\leq 2d\rho^d/N$, \eqref{affine_energy} gives
    \begin{equation}\label{affine_energy_strip}
        \frac{F_\epsilon(\ell,Q_\rho^2)}{\rho^d}\leq C|z|^2,\qquad \frac{F_\epsilon\left(\ell,\{\mathrm{dist}(\cdot,Q_\rho^c)\leq\rho/N\}\times Q_\rho\right)}{\rho^d}\leq \frac{C|z|^2}{N}.
    \end{equation}

    We now estimate $F_\epsilon(v_i,\Lambda_\epsilon)$. For $x,y\in Q_\rho$ we have
\begin{eqnarray*}\label{cutoff_identity}
                \nonumber v_i(x)-v_i(y) &= & \ \psi_i(x)\left(u_\epsilon(x)-u_\epsilon(y)\right)\\
        &&+\left(1-\psi_i(x)\right)\left(\ell(x)-\ell(y)\right)+\left(\psi_i(x)-\psi_i(y)\right)\varphi_\epsilon(y),    
        \end{eqnarray*}
    so that, by the convexity of $t\mapsto t^2$ applied to the first two terms, with weights $\psi_i(x)$ and $1-\psi_i(x)$, and by $(a+b)^2\leq 2a^2+2b^2$,
    \begin{equation}\label{cutoff_jensen}
        |v_i(x)-v_i(y)|^2\leq 2|u_\epsilon(x)-u_\epsilon(y)|^2+2|\ell(x)-\ell(y)|^2+2|\psi_i(x)-\psi_i(y)|^2|\varphi_\epsilon(y)|^2.
    \end{equation}

    The interactions at distance larger than $w$ are negligible. Indeed, by Lemma \ref{l_2_large_interactions}, applied with $A=Q_\rho$, $r=w$ and $f=v_i$,
    \[
        \frac{F_\epsilon\left(v_i,Q_\rho^2\cap\{|x-y|\geq w\}\right)}{\rho^d}\leq C\beta(1-s)N^{4s}\rho^{2(1-s)}\frac{\|v_i\|^2_{L^2(Q_\rho)}}{\rho^{d+2}},
    \]
    and, since $\|v_i\|^2_{L^2(Q_\rho)}\leq 2d|z|^2\rho^{d+2}+2\|\varphi_\epsilon\|^2_{L^2(Q_\rho)}$, by \eqref{blowup_convergence}, and the fact that $(1-s)\to0$, we obtain
    \begin{equation}\label{far_interactions}
        \lim\limits_{\epsilon\to0}\,\max\limits_{1\leq i\leq N}\frac{F_\epsilon\left(v_i,Q_\rho^2\cap\{|x-y|\geq w\}\right)}{\rho^d}=0.
    \end{equation}

    We then estimate the contribution of the cut-off part. Let
    \[
        R_i\coloneqq\frac{1-s}{\rho^d}\int_{Q_\rho}\int_{Q_\rho}a\left(\frac x\epsilon,\frac y\epsilon\right)\frac{|\psi_i(x)-\psi_i(y)|^2|\varphi_\epsilon(y)|^2}{|x-y|^{d+2s}}dxdy.
    \]
    By \eqref{properties_of_the_cutoff_liminf}, $|\psi_i(x)-\psi_i(y)|\leq\min\{1,2|x-y|/w\}$; hence, integrating in polar coordinates, for $s\geq1/2$ and every $y\in Q_\rho$,
    \[
        (1-s)\int_{\R^d}\min\Big\{1,\frac{4|x-y|^2}{w^2}\Big\}\frac{dx}{|x-y|^{d+2s}}\leq 4\sigma_{d-1}N^4\rho^{-2s},
    \]
    so that
    \begin{equation}\label{cutoff_error}
        \max\limits_{1\leq i\leq N}R_i\leq 4\beta\sigma_{d-1}N^4\rho^{2(1-s)}\frac{\|\varphi_\epsilon\|^2_{L^2(Q_\rho)}}{\rho^{d+2}}\longrightarrow 0,
    \end{equation}
    by \eqref{blowup_convergence}.

    We decompose the remaining interactions. For $(x,y)\in Q_\rho^2$, we have $v_i=u_\epsilon$ on $Q_{\rho,i}\times Q_{\rho,i}$ and $v_i=\ell$ on $\left(Q_\rho\setminus Q_{\rho,i-1}\right)^2$, while on 
    \[
        S_{\epsilon,i}\coloneqq \Lambda_\epsilon\cap\{|x-y|<w\}\setminus\Big(Q_{\rho,i}^2\cup(Q_\rho\setminus Q_{\rho,i-1})^2\Big)
    \]
    we use \eqref{cutoff_jensen}. Since $Q_\rho\setminus Q_{\rho,i-1}\subset\{\mathrm{dist}(\cdot,Q_\rho^c)\leq\rho/N\}$, by \eqref{affine_energy_strip}, \eqref{far_interactions}, and \eqref{cutoff_error} we obtain
    \begin{equation}\label{cutoff_decomposition}
        \frac{F_\epsilon(v_i,\Lambda_\epsilon)}{\rho^d}\leq \frac{F_\epsilon(u_\epsilon,\Lambda_\epsilon)}{\rho^d}+\frac{C|z|^2}{N}+2\frac{F_\epsilon(u_\epsilon,S_{\epsilon,i})+F_\epsilon(\ell,S_{\epsilon,i})}{\rho^d}+2R_i+o_\epsilon(1).
    \end{equation}

    It remains to choose the index $i$. A pair $(x,y)\in\Lambda_\epsilon$ with $|x-y|<w$ belongs to $S_{\epsilon,i}$ if and only if 
    \begin{align*}
        \min\{\mathrm{dist}(x,Q_\rho^c),\mathrm{dist}(y,Q_\rho^c)\}\leq i\,w,\qquad
        \max\{\mathrm{dist}(x,Q_\rho^c),\mathrm{dist}(y,Q_\rho^c)\}>(i-1)w;
    \end{align*}
    that is, if and only if $i$ belongs to a half-open interval of length smaller than $2$, because $\mathrm{dist}(\cdot,Q_\rho^c)$ is $1$-Lipschitz and $|x-y|<w$. Hence, every such pair belongs to at most two of the sets $S_{\epsilon,1},\ldots,S_{\epsilon,N}$, and, by \eqref{blowup_energy_bound} and \eqref{affine_energy_strip},
    \[
        \sum_{i=1}^N2\frac{F_\epsilon(u_\epsilon,S_{\epsilon,i})+F_\epsilon(\ell,S_{\epsilon,i})}{\rho^d}\leq 4\frac{F_\epsilon(u_\epsilon,\Lambda_\epsilon)}{\rho^d}+4\frac{F_\epsilon(\ell,Q_\rho^2)}{\rho^d}\leq C\left(C_0+|z|^2\right).
    \]
    We may then choose $i_\epsilon\in\{1,\ldots,N\}$ for which the corresponding summand is smaller than $C\left(C_0+|z|^2\right)/N$, and \eqref{affine_datum_estimate} follows from \eqref{cutoff_decomposition} and \eqref{blowup_energy_bound}. Finally, the same argument applies if $\Lambda_\epsilon$ is replaced by $Q_\rho^2$. Choosing $i_\epsilon$ so as to minimize the sum of the two corresponding summands, the conclusion holds for both, up to possibly increasing the constant $C$ in the statement.
\end{proof}

We now periodize the competitor. From now on, $v_\epsilon\coloneqq v_{\epsilon,i_\epsilon}$ denotes the sequence given by Lemma \ref{l_change_of_boundary_data}, and we set
\begin{equation}\label{definition_varphi}
    \varphi_\epsilon\coloneqq v_\epsilon-\ell,
\end{equation}
which vanishes on $\{x\in Q_\rho\mid \mathrm{dist}(x,Q_\rho^c)\leq\rho/N\}$, allowing us to extend it to $\R^d$, by $\rho$-periodicity, while keeping its fractional energy finite on every bounded measurable set. Since $\rho/\epsilon\in\N$, up to a translation of the lattice $\epsilon\Z^d$ we may assume that
\begin{equation}\label{cell_decomposition}
    Q_\rho=\bigcup_{k\in K}Q_{\epsilon,k},\qquad \# K=\left(\frac\rho\epsilon\right)^d,
\end{equation}
for some $K\subset\Z^d$; note that \eqref{cell_average_a} holds for any pair of cubes of side $\epsilon$, hence for the translated ones. For $\psi\in L^2_{\rm loc}(\R^d)$ which is $\rho$-periodic, and for $M\geq0$, we set
\begin{equation}\label{periodic_energy}
    E^M_\epsilon(\psi)\coloneqq\frac{1-s}{\rho^d}\int_{Q_\rho}\int_{\{M\epsilon\leq|x-y|\leq \rho/N\}}a\left(\frac x\epsilon,\frac y\epsilon\right)
    \dfrac{|\langle z,x-y\rangle+\psi(x)-\psi(y)|^2}{|x-y|^{d+2s}}\,dydx,   
\end{equation}
where $y$ ranges in $\R^d$; here $\epsilon$ is small enough for $M\epsilon<\rho/N$ to hold. The integrand in \eqref{periodic_energy} is invariant under the simultaneous translation $(x,y)\mapsto(x+\rho m,y+\rho m)$, for $m\in\Z^d$: this is clear for the kernel, which only depends on $x-y$, and for $\psi$, which is $\rho$-periodic, while for the coefficient it follows from $\rho/\epsilon\in\N$, since
\begin{equation}\label{diagonal_periodicity}
    a\Big(\frac{x+\rho m}\epsilon,\frac{y+\rho m}\epsilon\Big)=a\left(\frac x\epsilon+\frac\rho\epsilon m,\frac y\epsilon+\frac\rho\epsilon m\right)=a\left(\frac x\epsilon,\frac y\epsilon\right).
\end{equation}
Writing $y=x+h$, this means that the integrand is $\rho$-periodic in $x$ for every fixed $h$, so that the domain $Q_\rho$ of the outer integral in \eqref{periodic_energy} may be replaced by any of its translates. If $\psi$ is $\epsilon$-periodic, then \eqref{diagonal_periodicity} holds with $\epsilon$ in place of $\rho$, and \eqref{cell_decomposition} gives
\begin{equation}\label{periodic_energy_cell}
E^M_\epsilon(\psi)=\frac{1-s}{\epsilon^d}\int_{Q_{\epsilon,k}}\int_{\{M\epsilon\leq|x-y|\leq \rho/N\}}a\left(\frac x\epsilon,\frac y\epsilon\right)
    \dfrac{|\langle z,x-y\rangle+\psi(x)-\psi(y)|^2}{|x-y|^{d+2s}}dydx,
\end{equation}
for every $k\in K$.

\begin{lemma}[Periodization]\label{l_periodization}
    Let $M\geq0$, and let $\varphi_\epsilon$ be as in \eqref{definition_varphi}. Then the function
    \begin{equation}\label{definition_varphi_bar}
        \overline\varphi_\epsilon\coloneqq\left(\frac\epsilon\rho\right)^d\sum_{k\in K}\varphi_\epsilon(\cdot-\epsilon k)
    \end{equation}
    is $\epsilon$-periodic, and satisfies $\|\overline\varphi_\epsilon\|_{L^2(Q_\rho)}\leq\|\varphi_\epsilon\|_{L^2(Q_\rho)}$ and
    \begin{equation}\label{periodization_estimate}
        \frac{F_\epsilon\left(v_\epsilon,A_{1,\epsilon}\cap Q_\rho^2\right)}{\rho^d}\geq E^M_\epsilon(\overline\varphi_\epsilon)-\frac{C|z|^2}{N},
    \end{equation}
    where $C=C(d,\beta)>0$ does not depend on $M$ and $N$.
\end{lemma}
\begin{proof}
    Since the integrand is nonnegative, we have
    \begin{equation}\label{periodization_split}
        \frac{F_\epsilon\left(v_\epsilon,A_{1,\epsilon}\cap Q_\rho^2\right)}{\rho^d}\geq\frac{F_\epsilon\left(v_\epsilon,A_{1,\epsilon}\cap Q_\rho^2\cap\{|x-y|\leq \rho/N\}\right)}{\rho^d}=E^M_\epsilon(\varphi_\epsilon)-T,
    \end{equation}
    where
$$
        T\coloneqq\frac{1-s}{\rho^d}\int_{Q_\rho}\int_{\{M\epsilon\leq|x-y|\leq \rho/N\}\setminus Q_\rho}a\left(\frac x\epsilon,\frac y\epsilon\right)
        \dfrac{|\langle z,x-y\rangle+\varphi_\epsilon(x)-\varphi_\epsilon(y)|^2}{|x-y|^{d+2s}}dydx.
$$    
    If $x\in Q_\rho$, $y\notin Q_\rho$ and $|x-y|\leq \rho/N$, then $\mathrm{dist}(x,Q_\rho^c)\leq|x-y|\leq\rho/N$, so that $\varphi_\epsilon(x)=0$. Moreover, for almost every such $y$ there exists $m\in\Z^d\setminus\{0\}$ with $y\in Q_\rho+\rho m$, and, since $x$ does not belong to the interior of $Q_\rho+\rho m$,
    \[
        \mathrm{dist}\left(y,\left(Q_\rho+\rho m\right)^c\right)\leq|x-y|\leq\frac\rho N,
    \]
    whence $\varphi_\epsilon(y)=\varphi_\epsilon(y-\rho m)=0$ as well. Therefore, setting $\Sigma\coloneqq\{x\in Q_\rho\mid\mathrm{dist}(x,Q_\rho^c)\leq \rho/N\}$ and using \eqref{coercivity_boundedness_a}, $|\langle z,x-y\rangle|\leq|z||x-y|$ and $|\Sigma|\leq 2d\rho^{d}/N$, we obtain
    \begin{equation}\label{periodization_boundary}
        T\leq\beta|z|^2\frac{1-s}{\rho^d}\int_{\Sigma}\int_{\{|h|\leq \rho/N\}}\dfrac{dh}{|h|^{d+2s-2}}dx=\beta|z|^2\sigma_{d-1}\frac{|\Sigma|}{\rho^d}\dfrac{(\rho/N)^{2(1-s)}}{2}
        \leq\dfrac{C|z|^2}{N},
    \end{equation}
    where we also used that $(\rho/N)^{2(1-s)}\leq1$.

Changing variables by a translation $\epsilon k$ in \eqref{periodic_energy}, and using the $1$-periodicity of $a$ in both variables together with the fact that the outer domain in \eqref{periodic_energy} may be replaced by any of its translates, one gets
    \begin{equation}\label{periodization_translation}
        E^M_\epsilon(\varphi^k_\epsilon)=E^M_\epsilon(\varphi_\epsilon),\qquad \varphi^k_\epsilon\coloneqq\varphi_\epsilon(\cdot-\epsilon k),
    \end{equation}
    for every $k\in\Z^d$.

    The function $\overline\varphi_\epsilon$ is $\epsilon$-periodic. Indeed, for $j=1,\ldots,d$ and $x\in\R^d$,
    \[
        \overline\varphi_\epsilon(x+\epsilon e_j)=\left(\frac\epsilon\rho\right)^d\sum_{k\in K}\varphi_\epsilon\left(x-\epsilon(k-e_j)\right)=\overline\varphi_\epsilon(x),
    \]
    since $\varphi_\epsilon$ is $\rho$-periodic.

    Finally, $E^M_\epsilon$ is convex, being the integral of the square of an affine function of $\psi$, and, by \eqref{definition_varphi_bar}, $\overline\varphi_\epsilon$ is a convex combination of the functions $\varphi^k_\epsilon$, $k\in K$. Hence
    \[
        E^M_\epsilon(\overline\varphi_\epsilon)\leq\left(\frac\epsilon\rho\right)^d\sum_{k\in K}E^M_\epsilon(\varphi^k_\epsilon)=E^M_\epsilon(\varphi_\epsilon),
    \]
    and \eqref{periodization_estimate} follows from \eqref{periodization_split} and \eqref{periodization_boundary}. The same convexity argument applied to $\|\cdot\|^2_{L^2(Q_\rho)}$ gives
    \[
        \|\overline\varphi_\epsilon\|^2_{L^2(Q_\rho)}\leq\left(\frac\epsilon\rho\right)^d\sum_{k\in K}\|\varphi^k_\epsilon\|^2_{L^2(Q_\rho)}=\|\varphi_\epsilon\|^2_{L^2(Q_\rho)},
    \]
    where we used that $\|\varphi^k_\epsilon\|_{L^2(Q_\rho)}=\|\varphi_\epsilon\|_{L^2(Q_\rho-\epsilon k)}=\|\varphi_\epsilon\|_{L^2(Q_\rho)}$, again by $\rho$-periodicity.
\end{proof}

Now, we show that the interactions of $\overline\varphi_\epsilon$ with itself are negligible, for $M$ large.

\begin{lemma}\label{l_corrector_long_range}
    Let $M\geq4\sqrt d$ and let $\overline\varphi_\epsilon$ be as in \eqref{definition_varphi_bar}. Then
    \begin{equation}\label{corrector_long_range}
        P_\epsilon\coloneqq\frac{1-s}{\rho^d}\int_{Q_\rho}\int_{\{M\epsilon\leq|x-y|\leq\rho/N\}}a\left(\frac x\epsilon,\frac y\epsilon\right)\dfrac{|\overline\varphi_\epsilon(x)-\overline\varphi_\epsilon(y)|^2}{|x-y|^{d+2s}}dydx\leq\dfrac{C}{M^{2s}},
    \end{equation}
    where $C>0$ does not depend on $\epsilon$, $M$ and $N$.
\end{lemma}
\begin{proof}
    By Lemma \ref{l_change_of_boundary_data} and Lemma \ref{l_periodization}, both applied with $M=0$, and by \eqref{blowup_energy_bound}, we have $E^0_\epsilon(\overline\varphi_\epsilon)\leq C$. Hence, by \eqref{periodic_energy_cell}, $(a+b)^2\leq 2a^2+2b^2$, \eqref{coercivity_boundedness_a} and \eqref{sphere_average}, for every $k\in K$ we obtain
    \[
        \frac{1-s}{\epsilon^d}\int_{Q_{\epsilon,k}}\int_{Q_{\epsilon,k}}\dfrac{|\overline\varphi_\epsilon(x)-\overline\varphi_\epsilon(y)|^2}{|x-y|^{d+2s}}dydx\leq\frac 2\alpha\left(E^0_\epsilon(\overline\varphi_\epsilon)+\beta\dfrac{\sigma_{d-1}}{2d}|z|^2\right)\leq C,
    \]
    where we also used that $\sqrt d\,\epsilon\leq\rho/N$ for $\epsilon$ small. Since $|x-y|^{-d-2s}\geq(\sqrt d\,\epsilon)^{-d-2s}$ for $x,y\in Q_{\epsilon,k}$, this gives
    \begin{equation}\label{cell_L2_bound}
        \int_{Q_{\epsilon,k}}\int_{Q_{\epsilon,k}}|\overline\varphi_\epsilon(x)-\overline\varphi_\epsilon(y)|^2dydx\leq C\dfrac{\epsilon^{2d+2s}}{1-s}.
    \end{equation}
    On the other hand, by \eqref{periodic_energy_cell} and by the $\epsilon$-periodicity of $\overline\varphi_\epsilon$, writing $y=y'+\epsilon k$ we get
    \begin{equation}\label{corrector_split}
        P_\epsilon\leq\frac{\beta(1-s)}{\epsilon^d}\sum_{k\in\Z^d}m_k\int_{Q_{\epsilon,0}}\int_{Q_{\epsilon,0}}|\overline\varphi_\epsilon(x)-\overline\varphi_\epsilon(y')|^2dy'dx,
    \end{equation}
    where
    \[
        m_k\coloneqq\sup\big\{|x-y|^{-d-2s}\mid x\in Q_{\epsilon,0},\ y\in Q_{\epsilon,k},\ |x-y|\geq M\epsilon\big\},
    \]
    with the convention that $m_k=0$, if the above set is empty. For $x\in Q_{\epsilon,0}$ and $y\in Q_{\epsilon,k}$, we have $\epsilon(|k|-\sqrt d)\leq|x-y|\leq\epsilon(|k|+\sqrt d)$, so that $m_k=0$ unless $|k|\geq M-\sqrt d\geq M/2\geq2\sqrt d$, and in that case $|x-y|\geq\epsilon|k|/2$, whence $m_k\leq\left(\epsilon|k|/2\right)^{-d-2s}$. Moreover, if $h\in Q_{1,k}$ and $|k|\geq2\sqrt d$, then $|h|\leq|k|+\sqrt d\leq2|k|$ and $|h|\geq|k|-\sqrt d\geq M/4$, so that
    \begin{eqnarray*}
        \sum_{|k|\geq M/2}\dfrac{1}{|k|^{d+2s}}&\leq& 2^{d+2s}\sum_{|k|\geq M/2}\int_{Q_{1,k}}\dfrac{dh}{|h|^{d+2s}}
        \leq 2^{d+2s}\int_{\{|h|\geq M/4\}}\dfrac{dh}{|h|^{d+2s}}\\&=&\dfrac{2^{d+2s}\sigma_{d-1}}{2s}\left(\dfrac M4\right)^{-2s}.        
    \end{eqnarray*}
    Therefore,
    \begin{equation}\label{kernel_sum_bound}
        \sum_{k\in\Z^d}m_k\leq\left(\dfrac2\epsilon\right)^{d+2s}\sum_{|k|\geq M/2}\dfrac{1}{|k|^{d+2s}}\leq\dfrac{C}{\epsilon^{d+2s}M^{2s}},
    \end{equation}
    Putting together \eqref{cell_L2_bound}, \eqref{corrector_split}  and \eqref{kernel_sum_bound}, the claim follows.
\end{proof}

We are finally in a position to estimate the density of $\mu$.

\begin{proposition}\label{p_liminf_long_range}
    For $\mathscr{L}^d$-almost every $x_0\in\Omega$,
    \begin{equation}\label{density_bound}
        \frac{d\mu}{d\mathscr{L}^d}(x_0)\geq(1-\lambda)\dfrac{\sigma_{d-1}}{2d}\,\overline a\,|\nabla u(x_0)|^2.
    \end{equation}
    Consequently,
    \begin{equation}\label{liminf_long_range}
        \liminf\limits_{\epsilon\to 0}F_\epsilon(u_\epsilon,A_{1,\epsilon})\geq(1-\lambda)\dfrac{\sigma_{d-1}}{2d}\,\overline a\,\|\nabla u\|^2_{L^2(\Omega)}.
    \end{equation}
\end{proposition}
\begin{proof}
    Let $x_0$, $\rho$, $z$, and $\ell$ be as above, and fix $M\geq4\sqrt d$ and $N\geq2$. By Lemma \ref{l_change_of_boundary_data} and Lemma \ref{l_periodization},
    \begin{equation}\label{conclusion_step_1}
        \liminf\limits_{\epsilon\to 0}\frac{\mu_\epsilon(Q_\rho)}{\mathscr{L}^d(Q_\rho)}\geq\liminf\limits_{\epsilon\to 0}E^M_\epsilon(\overline\varphi_\epsilon)-\dfrac{C}{N}.
    \end{equation}
    Expanding the square in \eqref{periodic_energy} and using the Cauchy--Schwarz inequality, we get
    \begin{align*}
        E^M_\epsilon(\overline\varphi_\epsilon)&\geq\frac{1-s}{\rho^d}\int_{Q_\rho}\int_{\{M\epsilon\leq|x-y|\leq\rho/N\}}a\left(\frac x\epsilon,\frac y\epsilon\right)\dfrac{|\langle z,x-y\rangle|^2}{|x-y|^{d+2s}}dydx\\
        &\hphantom{\geq{}}-2\frac{1-s}{\rho^d}\int_{Q_\rho}\int_{\{M\epsilon\leq|x-y|\leq\rho/N\}}a\left(\frac x\epsilon,\frac y\epsilon\right)
        \dfrac{|\langle z,x-y\rangle|\,|\overline\varphi_\epsilon(x)-\overline\varphi_\epsilon(y)|}{|x-y|^{d+2s}}dydx\\
        &\geq E^M_\epsilon(0)-2\sqrt{E^M_\epsilon(0)P_\epsilon},
    \end{align*}
    where we discarded the nonnegative term $P_\epsilon$ and recognized $E^M_\epsilon(0)$ and $P_\epsilon$ as mixed terms. Hence, by Lemma \ref{l_corrector_long_range},
    \begin{equation}\label{conclusion_step_2}
        E^M_\epsilon(\overline\varphi_\epsilon)\geq E^M_\epsilon(0)-2\sqrt{E^M_\epsilon(0)P_\epsilon}\geq E^M_\epsilon(0)-\dfrac{C}{M^s},
    \end{equation}
    where we also used $E^M_\epsilon(0)\leq\beta\sigma_{d-1}|z|^2/(2d)$.

    It remains to compute the limit of $E^M_\epsilon(0)$. Applying Lemma \ref{l_averaging_large_scale} with $A=Q_\rho$, $\xi=z$, $r_1=M\epsilon$ and $r_2=\rho/N$, and using $|\partial_\epsilon Q_\rho|\leq C\rho^{d-1}\epsilon$, we get
    \[
        \Big|E^M_\epsilon(0)-\overline a\,\dfrac{\sigma_{d-1}}{2d}|z|^2\Big(\left(\dfrac\rho N\right)^{2(1-s)}-(M\epsilon)^{2(1-s)}\Big)\Big|\leq C\beta|z|^2\Big(\dfrac{1-s}M+\dfrac\epsilon\rho\Big),
    \]
    which vanishes as $\epsilon\to0$. Since $M$ and $N$ are fixed, $\left(\rho/N\right)^{2(1-s)}\to1$ and $(M\epsilon)^{2(1-s)}\to\lambda$, by \eqref{rho_asymptotics} and by the choice of $\lambda$. We obtain
    \begin{equation}\label{conclusion_step_3}
        \lim\limits_{\epsilon\to 0}E^M_\epsilon(0)=(1-\lambda)\dfrac{\sigma_{d-1}}{2d}\,\overline a\,|z|^2.
    \end{equation}
    Putting together \eqref{conclusion_step_1}, \eqref{conclusion_step_2} and \eqref{conclusion_step_3}, and recalling the choice of $\rho$, we obtain
    \[
        \frac{d\mu}{d\mathscr{L}^d}(x_0)\geq(1-\lambda)\dfrac{\sigma_{d-1}}{2d}\,\overline a\,|z|^2-\dfrac{C}{M^s}-\dfrac CN,
    \]
    with $C$ independent of $M$ and $N$. Letting $M\to+\infty$ and $N\to+\infty$ gives \eqref{density_bound}. Since
    \[
        \liminf\limits_{\epsilon\to 0}F_\epsilon(u_\epsilon,A_{1,\epsilon})=\liminf\limits_{\epsilon\to 0}\mu_\epsilon(\Omega)\geq\mu(\Omega)\geq\int_\Omega\frac{d\mu}{d\mathscr{L}^d}(x)dx.
    \]
     we obtain \eqref{liminf_long_range}.
\end{proof}

Putting together Lemma \ref{l_short_range} and Proposition \ref{p_liminf_long_range}, we conclude that
\begin{equation}\label{liminf_conclusion}
    \liminf\limits_{\epsilon\to 0}F_\epsilon(u_\epsilon)\geq\dfrac{\sigma_{d-1}}{2d}\Big((1-\lambda)\overline a\|\nabla u\|^2_{L^2(\Omega)}+\lambda F_{\rm hom}(u)\Big),
\end{equation}
which is the $\Gamma$-liminf inequality.

\section{Limsup inequality}\label{sec:limsup}
To prove the $\Gamma$-limsup inequality, 
it is sufficient to construct recovery sequences for piecewise-affine target functions; the general case then follows by a standard density argument. Moreover, we can suppose that $\Omega$ is connected.
Indeed, otherwise it is sufficient to make the construction in each connected component of $\Omega$ separately. Since different connected components have a strictly positive distance, the interactions of recovery sequences between different components are asymptotically negligible.

\smallskip
We consider a function $u\in H^1(\Omega)$ for which there exists a finite collection of $d$-simplices $\{\Delta_i\}_{i\in I}$, covering $\Omega$ and overlapping only on a $(d-1)$-common face, such that
\[
u(x)=\inp*{\xi_i}{x}+m_i
\]for $x\in\Delta_i\cap\Omega$ and some $m_i\in \mathbb{R}$, $\xi_i\in\mathbb{R}^d$. 
Note that such a $u$ 
is indeed Lipschitz continuous on the whole $\Omega$.

We fix $\delta>0$ and construct an approximate recovery sequence for $u$.
For each $i\in I$, consider a $1$-periodic function $\phi_i\in\mathcal{C}^{\infty}(\mathbb{R}^d)$, such that 
\begin{equation}\label{choice_phi_i}
    \int_{(0,1)^d} a(x,x)|\xi_i+\nabla \phi_i(x)|^2 dx \leq \inp*{A_{\rm hom}\xi_i}{\xi_i} + \delta.
\end{equation}
Moreover, for each $a>0$ and each $i\in I$, we define 
\begin{equation}
    \Delta_i^a\coloneqq\{x\in\Delta_i\cap\Omega\mid \mathrm{dist}(x,(\Delta_i)^c \cup\Omega^c)>a\}.
\end{equation}
Then, for each $\epsilon>0$, and for each $i\in I$, we choose a smooth function $\psi_{\epsilon,i}\in\mathcal{C}^{\infty}_c(\mathbb{R}^d)$, such that
\begin{equation}
    \label{properties of the cutoff}
    \1_{\Delta_i^{2\epsilon}} \leq \psi_{\epsilon,i}\leq \1_{\Delta_i^{\epsilon}}, \qquad \|\nabla \psi_{\epsilon,i}\|_\infty \leq 2/\epsilon, \qquad
        \|\nabla^2 \psi_{\epsilon,i}\|_\infty \leq 4/\epsilon^2.
\end{equation}
Finally, we define the approximate recovery sequence $(u_\epsilon)$ as
\begin{equation}\label{definition_recovery}
    u_\epsilon(x)\coloneqq u(x)+\epsilon\sum\limits_{i\in I}\psi_{\epsilon,i}(x)\phi_i\left(\frac x \epsilon\right),\qquad x\in\Omega.
\end{equation}

First of all, note that $u-u_\epsilon$ is in $\mathcal{C}^{\infty}(\Omega)$, and there exists $C>0$ such that $\|u-u_\epsilon\|_\infty\leq C\epsilon$ and $\|\nabla(u-u_\epsilon)\|_\infty\leq C$, so that $(u_\epsilon)\subset H^1(\Omega)$, $u_\epsilon\to u$ in $L^2(\Omega)$ and $(u_\epsilon)$ is bounded in $H^1(\Omega)$. 
We claim that 
\begin{equation}
    \limsup\limits_{\epsilon\to 0}F_\epsilon(u_\epsilon)\leq \frac{\sigma_{d-1}}{2d}\Big((1-\lambda)\overline{a}\|\nabla u\|_{L^2(\Omega)}^2+\lambda F_{\rm hom}(u)\Big)+C\delta,
\end{equation}
Before starting the computations, we divide the interactions in $\Omega\times\Omega$ into several regions, and we discuss which ones are negligible.

Consider two auxiliary parameters $r\,(\coloneqq r(\epsilon))$ and $\eta\,(\coloneqq \eta(\epsilon))$, satisfying 
\begin{align}
    \eta>0,\qquad \lim\limits_{\epsilon\to 0} \eta =0, \qquad \lim\limits_{\epsilon\to 0} \epsilon^\eta = 0,\quad
\hbox{ and }\quad    r > \epsilon, \qquad \lim\limits_{\epsilon\to 0} r =0, \qquad \lim\limits_{\epsilon\to 0}\frac{1-s}{r^2}=0.
\end{align}
Note that the last condition implies that $\lim\limits_{\epsilon\to 0} r^{2(1-s)} =1$.

For each $\epsilon>0$, we divide $\Omega^2$ into the following sets
\begin{align}
  \label{set_A} A_{1,\epsilon}&\coloneqq\left\{(x,y)\in\Omega^2\mid |x-y|\geq r\right\}\\
  \label{set_B} A_{2,\epsilon}&\coloneqq\left\{(x,y)\in\Omega^2\mid \epsilon<|x-y|<r,\, x\in\left(\cup_{i\in I}\Delta_i^{3r}\right)^c\right\}\\
  \label{set_C} A_{3,\epsilon}&\coloneqq\left\{(x,y)\in\Omega^2\mid \epsilon<|x-y|<r,\, x\in\cup_{i\in I}\Delta_i^{3r}\right\}\\
  \label{set_D} A_{4,\epsilon}&\coloneqq\left\{(x,y)\in\Omega^2\mid \epsilon^{1+\eta}\leq |x-y|\leq \epsilon\right\}\\
  \label{set_E} A_{5,\epsilon}&\coloneqq\left\{(x,y)\in\Omega^2\mid |x-y|<\epsilon^{1+\eta},\, x\in\cup_{i\in I}\Delta_i^{3\epsilon}\right\}\\
  \label{set_F} A_{6,\epsilon}&\coloneqq\left\{(x,y)\in\Omega^2\mid |x-y|<\epsilon^{1+\eta},\, x\in\left(\cup_{i\in I}\Delta_i^{3\epsilon}\right)^c\right\}.
\end{align}

Recalling the definition of the truncated functionals given in \eqref{truncatedfunctional}, we note that for all $\epsilon>0$ and all $v\in L^2(\Omega)$, it holds
\begin{equation}
    F_\epsilon(v)=\sum_{k=1}^6 F_{\epsilon}(v,A_{k,\epsilon}).
\end{equation}
We show that the contributions coming from sets $A_{1,\epsilon},A_{2,\epsilon},A_{4,\epsilon}$ and $A_{6,\epsilon}$ are negligible.
By Lemma \ref{l_2_large_interactions}, since $(u_\epsilon)$ is bounded in $L^2(\Omega)$, it follows that 
\begin{equation}\label{porzione_A1}
    \lim\limits_{\epsilon\to 0} F_{\epsilon}(u_\epsilon,A_{1,\epsilon})=0.
\end{equation}
By construction, the sequence $(u_\epsilon)$ is uniformly Lipschitz. Moreover,
\[
A_{2,\epsilon}\cup A_{6,\epsilon}\subset \Big\{(x,y)\in \Omega^2\mid x\in\Big(\bigcup_{i\in I}\Delta_i^{3r}\Big)^c\Big\},
\]
and
    $\big|\big(\bigcup_{i\in I}\Delta_i^{3r}\big)^c\big| \to 0$ as $\epsilon \to 0$.
We then apply Lemma \ref{w_1_infty_close_interactions} and obtain
\begin{equation}\label{porzione_A2_A6}
    \lim\limits_{\epsilon \to 0} F_{\epsilon}(u_\epsilon,A_{2,\epsilon})=\lim\limits_{\epsilon \to 0} F_{\epsilon}(u_\epsilon,A_{6,\epsilon})=0.
\end{equation}

To deal with the contribution coming from $A_{4,\epsilon}$, we use $\|\nabla u_\epsilon\|_{L^2(\Omega)} \leq M$ and apply Theorem \ref{BBM_theorem1}, with 
\begin{equation}
    \rho_\epsilon(z)=\dfrac{(1-s)}{|z|^{d+2s-2}} \1_{\{\epsilon^{1+\eta}\leq |z|\leq \epsilon\}}(z).
\end{equation}
We find that there exists $C=C(\Omega)>0$ such that 
\begin{align}
    F_\epsilon(u_\epsilon,A_{4,\epsilon})&\leq C \|\nabla u_\epsilon\|_{L^2(\Omega)}^2\|\rho_\epsilon\|_{L^1(\R^d)} 
    = \frac{CM \sigma_{d-1}}{2} \epsilon^{2(1-s)}(1-\epsilon^{2\eta(1-s)}) \\
    &\leq \frac{CM \sigma_{d-1}}{2} \min \{\epsilon^{2(1-s)},1-\epsilon^{2\eta(1-s)}\},
\end{align}
where we used $0 \leq \epsilon^{2(1-s)},1-\epsilon^{2\eta(1-s)} \leq 1$ for every $\epsilon > 0$.
Recalling that $\epsilon^{2(1-s)} \to \lambda \in [0,1]$, we obtain
\begin{equation}\label{porzione_A4}
    \limsup\limits_{\epsilon \to 0}F_\epsilon(u_\epsilon,A_{4,\epsilon}) \leq \limsup\limits_{\epsilon \to 0} \frac{CM \sigma_{d-1}}{2} \min 
    \big\{\epsilon^{2(1-s)},1-\epsilon^{2\eta(1-s)}\big\} =0
\end{equation}

We now turn to the relevant contributions. We first note that
\[
    |u_\epsilon(x)-u_\epsilon(y)|^2 = \left|\inp*{\xi_i}{x-y} +\epsilon \left(\phi_i\left(\frac x \epsilon\right)-\phi_i\left(\frac y \epsilon\right) \right)\right|^2,
\]
for $(x,y) \in A_{3,\epsilon}$. Using that
    $(a+b)^2 \leq (1+\gamma)a^2+\left(1+1/\gamma \right)b^2$ for $\gamma>0$, we obtain that
\[
    F_\epsilon(u_\epsilon,A_{3,\epsilon}) \leq (1+\gamma)F_\epsilon(u,A_{3,\epsilon})+\Big(1+\frac{1}{\gamma} \Big)F_\epsilon\left(\epsilon \phi_i\left(\frac \cdot \epsilon\right), A_{3,\epsilon} \right).
\]
By Lemma \ref{l_infty_large_interactions}, using that $\|\epsilon \phi_i\left(\cdot/ \epsilon\right) \|_\infty\leq C\epsilon$ for some $C > 0$ and for every $i \in I$, it follows that
\[
    \limsup\limits_{\epsilon \to 0} F_\epsilon(u_\epsilon,A_{3,\epsilon}) \leq (1+\gamma)\limsup\limits_{\epsilon \to 0}F_\epsilon(u,A_{3,\epsilon}).
\]
By the arbitrariness of $\gamma > 0$, we then conclude that 
\[
    \limsup\limits_{\epsilon \to 0} F_\epsilon(u_\epsilon,A_{3,\epsilon})\leq \limsup\limits_{\epsilon \to 0}F_\epsilon(u,A_{3,\epsilon}).
\]
We proceed with the explicit computation of the right-hand side of the previous equation. 
If $x\in\Delta_i^{3r}$ and $|x-y|<r$, then $y\in\Delta_i\cap\Omega$ and $u(x)-u(y)=\inp*{\xi_i}{x-y}$. Hence, applying Lemma \ref{l_averaging_large_scale} with $A=\Delta_i^{3r}$, $\xi=\xi_i$, $r_1=\epsilon$ and $r_2=r$, we get
\begin{align}
    F_\epsilon(u,A_{3,\epsilon}) &= (1-s)\sum_{i\in I}\int_{\Delta_i^{3r}}\int_{\{\epsilon<|x-y|<r\}} a\left(\frac{x}\epsilon,\frac{y}\epsilon\right)\dfrac{|\inp*{\xi_i}{x-y}|^2}{|x-y|^{d+2s}}dydx \\
    &= \dfrac{\sigma_{d-1}}{2d} (r^{2(1-s)}-\epsilon^{2(1-s)})\overline{a}\sum_{i \in I} |\Delta_i^{3r}|\,|\xi_i|^2+o_\epsilon(1),
\end{align}
where we used that the error term in \eqref{averaging_estimate} is bounded by
\[
    C\beta|\xi_i|^2\left((1-s)|\Delta_i^{3r}|+|\partial_\epsilon\Delta_i^{3r}|\right),
\]
which vanishes as $\epsilon\to0$, uniformly in $r$. Indeed, $1-s\to0$ and, setting $\delta_i(x)\coloneqq\mathrm{dist}\left(x,(\Delta_i)^c\cup\Omega^c\right)$, the function $\delta_i$ is $1$-Lipschitz, and therefore we have
\[
    \partial_\epsilon\Delta_i^{3r}\subset\left\{x\in\R^d\mid |\delta_i(x)-3r|\leq \sqrt d\,\epsilon\right\}.
\]
Since $\Delta_i\cap\Omega$ has Lipschitz boundary, there exist $C,t_0>0$ such that $|\{a<\delta_i\leq b\}|\leq C(b-a)$ for every $0\leq a<b\leq t_0$, so that $|\partial_\epsilon\Delta_i^{3r}|\leq 2\sqrt d\,C\epsilon$, for every $\epsilon$ small enough.
Therefore, it follows that 
\begin{align}
    \limsup\limits_{\epsilon \to 0}F_\epsilon(u_\epsilon,A_{3,\epsilon})&\leq \limsup\limits_{\epsilon \to 0}F_\epsilon(u,A_{3,\epsilon})\\ &=\dfrac{\sigma_{d-1}}{2d}\limsup\limits_{\epsilon \to 0} (r^{2(1-s)}-\epsilon^{2(1-s)})\overline{a}\sum_{i \in I} |\Delta_i^{3r}|\,|\xi_i|^2 \\
    &= \dfrac{\sigma_{d-1}}{2d}(1-\lambda) \overline{a} \sum_{i \in I} |\Delta_i \cap \Omega|\,|\xi_i|^2 \\
   \label{porzione_A3} &= \dfrac{\sigma_{d-1}}{2d}(1-\lambda) \overline{a} \|\nabla u\|_{L^2(\Omega)}^2.
\end{align}
We are left with the contribution coming from $A_{5,\epsilon}$. First, we substitute the periodic function $a(x,y)$ with its diagonal counterpart $a(x,x)$. To do so, we choose a modulus of continuity $\omega$ for $a$; namely, an increasing function $\omega : \R^+ \rightarrow \R^+$ such that  $\lim_{t \to 0} \omega(t) = \omega(0) = 0$,
for which
\[
    |a(x,y)-a(z,w)|\leq \omega(|(x,y)-(z,w)|),
\]
for every $(x,y),(z,w) \in \R^{2d}$. Now
\begin{align}
     F_\epsilon(u_\epsilon, A_{5,\epsilon}) &= (1-s)\int_{A_{5,\epsilon}} a\left(\frac{x}\epsilon,\frac{x}\epsilon\right)
\frac{|u_\epsilon(x)-u_\epsilon(y)|^2}{|x-y|^{d+2s}}dxdy  \\
 &\ \qquad + (1-s)\int_{A_{5,\epsilon}} \left[a\left(\frac{x}\epsilon,\frac{y}\epsilon\right)-a\left(\frac{x}\epsilon,\frac{x}\epsilon\right)\right]
\frac{|u_\epsilon(x)-u_\epsilon(y)|^2}{|x-y|^{d+2s}}dxdy\\
&=:I_{1,\epsilon}+I_{2,\epsilon}.
\end{align}

We estimate $I_{2,\epsilon}$ from above, applying Theorem \ref{BBM_theorem1} with
\[
    \rho_\epsilon(z)=\dfrac{(1-s)}{|z|^{d+2s-2}} \1_{\{|z|\leq 1\}}(z),
\]
obtaining
\begin{align}
 I_{2,\epsilon}&\leq (1-s)\int_{A_{5,\epsilon}} \omega\left(\dfrac{|x-y|}{\epsilon}\right)
\dfrac{|u_\epsilon(x)-u_\epsilon(y)|^2}{|x-y|^{d+2s}}dxdy \\
&\leq \omega(\epsilon^\eta)(1-s) \int_{A_{5,\epsilon}}\dfrac{|u_\epsilon(x)-u_\epsilon(y)|^2}{|x-y|^{d+2s}}dxdy \\
\label{modulo_di_continuita}&\leq \omega(\epsilon^\eta)\|\nabla u_\epsilon\|_{L^2(\Omega)}^2 \|\rho_\epsilon\|_{L^1(\R^d)} \leq C \omega(\epsilon^\eta),
\end{align}
where we used that $(u_\epsilon)$ is bounded in $H^1(\Omega)$ and that $(\rho_\epsilon)$ is bounded in $L^1(\R^d)$, by the usual integration in polar coordinates. 

As for $I_{1,\epsilon}$, for each $\epsilon>0$ and each $i\in I$, we use the first-order Taylor expansion of $\epsilon\phi_i(\cdot/\epsilon)$, that is
\begin{equation}
    \epsilon \phi_i\left(\frac y \epsilon\right)-\epsilon \phi_i\left(\frac x \epsilon\right)=\inp*{\nabla\phi_i\left(\frac x\epsilon\right)}{x-y}+R_{\epsilon,i}\left(x,y\right),
\end{equation}
with 
\begin{equation}
    |R_{\epsilon,i}\left(x,y\right)|\leq \max_{i\in I}\|\nabla^2\phi_i\|_\infty\frac{|x-y|^2}{\epsilon}.
\end{equation}
Using again $(a+b)^2 \leq (1+\gamma)a^2+\left(1+1/\gamma \right)b^2$, for $\gamma>0$ and that 
\[
A_{5,\epsilon}\subset\cup_{i\in I}\big(\Delta_i^{3\epsilon}\times \Delta_i^{2\epsilon}\big),\] 
it follows that
\begin{align}
  I_{1,\epsilon} &\leq (1+\gamma)(1-s)\sum_{i\in I}\int_{\Delta_i^{3\epsilon}}\int_{\Delta_i^{2\epsilon}}a\left(\frac x\epsilon,\frac x \epsilon\right)\dfrac{\left|\inp*{\xi_i+\nabla\phi_i\left(\frac x\epsilon\right)}{x-y}\right|^2}{|x-y|^{d+2s}}dydx\\
&\ \ \qquad +\Big(1+\frac 1\gamma\Big)\max_{i\in I}\|\nabla^2\phi_i\|_\infty\frac{1-s}{\epsilon^2}\int_{A_{5,\epsilon}}a\left(\frac x\epsilon,\frac x\epsilon\right)\dfrac{1}{|x-y|^{d+2s-4}}dxdy\\
&=:(1+\gamma)I_{3,\epsilon}+\Big(1+\frac 1\gamma\Big)I_{4,\epsilon}.
\end{align}
Continuing the computation, integrating in polar coordinates, we obtain
\begin{align}
    I_{4,\epsilon}&\leq\max_{i\in I}\|\nabla^2\phi_i\|_\infty\frac{\beta(1-s)}{\epsilon^2}\int_{\Omega}\int_{\mathbb{S}^{d-1}}\int_0^{\epsilon^{1+\eta}}t^{3-2s}dt d\sigma(\nu)dx\\
    &\leq\max_{i\in I}\|\nabla^2\phi_i\|_\infty\frac{\beta(1-s)\sigma_{d-1}|\Omega|}{2(4-2s)}\epsilon^{(1+\eta)(4-2s)-2}\\
   \label{o_piccolo_1} &\leq C(1-s)\epsilon^{(1+\eta)(4-2s)-2}=o_\epsilon(1),
\end{align}
since $1-s\to 0$ and $\epsilon^{(1+\eta)(4-2s)-2}$ is bounded, because $(1+\eta)(4-2s)-2>0$.

Regarding the term $I_{3,\epsilon}$, exploiting \eqref{sphere_average} and integrating in polar coordinates, we obtain
\begin{align}
    I_{3,\epsilon}&=(1-s)\sum_{i\in I}\int_{\Delta_i^{3\epsilon}}a\left(\frac x\epsilon,\frac x \epsilon\right)\int_{\mathbb{S}^{d-1}}\left|\inp*{\xi_i+\nabla\phi_i\left(\frac x\epsilon\right)}{\nu}\right|^2\int_0^{\epsilon}t^{1-2s}dt d\sigma(\nu)dx\\
    &=\frac{\sigma_{d-1}}{2d}\epsilon^{2(1-s)}\sum_{i\in I}\int_{\Delta_i^{3\epsilon}}a\left(\frac x\epsilon,\frac x\epsilon\right)\left|\xi_i+\nabla\phi_i\left(\frac x\epsilon\right)\right|^2 dx\\
   \label{pezzociccio} &\leq\frac{\sigma_{d-1}}{2d}\epsilon^{2(1-s)}\sum_{i\in I}|\Delta_i\cap\Omega|\,\left(\inp{A_{\rm hom}\xi_i}{\xi_i}+\delta\right).
\end{align}
Therefore, taking into account \eqref{modulo_di_continuita}, \eqref{o_piccolo_1}, \eqref{pezzociccio} and the choice of $\phi_i$ (see \eqref{choice_phi_i}), we obtain
\begin{align}
    \limsup\limits_{\epsilon\to 0}F_\epsilon(u_\epsilon,A_{5,\epsilon})&\leq \limsup\limits_{\epsilon\to 0}I_{1,\epsilon}+\limsup\limits_{\epsilon\to 0}I_{2,\epsilon}\\
    &\leq (1+\gamma)\limsup\limits_{\epsilon\to 0}I_{3,\epsilon}+\left(1+\frac 1 \gamma\right)\limsup\limits_{\epsilon\to 0}I_{4,\epsilon}+C\limsup\limits_{\epsilon\to 0}\omega(\epsilon^{\eta})\\
    &\leq (1+\gamma)\frac{\sigma_{d-1}}{2d}\lambda\left(F_{\rm hom}(u)+\delta|\Omega|\right).
\end{align}
where we used that $\epsilon^{\eta}\to 0$ and $\lim_{t\to 0^+}\omega(t)=0$.
Taking the infimum over $\gamma>0$, we obtain 
\begin{equation}\label{porzione_A5}
    \limsup\limits_{\epsilon\to 0}F_\epsilon(u_\epsilon,A_{5,\epsilon})\leq \frac{\sigma_{d-1}}{2d}\lambda\left(F_{\rm hom}(u)+\delta|\Omega|\right). 
\end{equation}
Gathering all the information from equations \eqref{porzione_A1}, \eqref{porzione_A2_A6}, \eqref{porzione_A4}, \eqref{porzione_A3}, and \eqref{porzione_A5}, we obtain
\begin{align}
    \limsup\limits_{\epsilon\to 0}F_\epsilon(u_\epsilon)&\leq \sum_{k=1}^6 \limsup\limits_{\epsilon\to 0}F_\epsilon(u_\epsilon,A_{k,\epsilon})\\
    &=\limsup\limits_{\epsilon\to 0}F_\epsilon(u_\epsilon,A_{3,\epsilon})+\limsup\limits_{\epsilon\to 0}F_\epsilon(u_\epsilon,A_{5,\epsilon})\\
    &\leq \dfrac{\sigma_{d-1}}{2d}(1-\lambda) \overline{a} \|\nabla u\|_{L^2(\Omega)}^2 + \frac{\sigma_{d-1}}{2d}\lambda F_{\rm hom}(u)+C\delta, 
\end{align}
where $C>0$ is some fixed constant, depending only on $d$, $\beta$ and $|\Omega|$. Repeating this construction for a sequence of $\delta\to0$ and using a diagonal argument, we obtain a recovery sequence for $u$, as desired.

Since $\Omega$ is an extension domain, any function $u\in H^1(\Omega)$ admits an extension $\overline{u}\in H^1_0(B_R(0))$, for some fixed positive radius $R>0$. Moreover, piecewise-affine functions are dense in $H^1_0(B_R(0))$ and hence, by restriction, they are also dense in $H^1(\Omega)$. As a consequence, since the functional 
\begin{equation}
    F_\lambda(u)=\dfrac{\sigma_{d-1}}{2d}\left((1-\lambda) \overline{a} \|\nabla u\|_{L^2(\Omega)}^2 + \lambda F_{\rm hom}(u)\right)
\end{equation}
is continuous in the strong topology of $H^1(\Omega)$, the $\Gamma$-$\limsup$ inequality follows as we proved it for a strongly dense subset of $H^1(\Omega)$.

\bigskip\noindent
{\bf Acknowledgements.} 
 Andrea Braides acknowledges funding by the Excellence Project MatMod@TOV (Grant Agreement No.~CUP\_E83C23000330006) of the
Department of Mathematics of the University of Rome Tor Vergata. The authors thank Giampiero Palatucci for valuable remarks.

\baselineskip=12.6pt
\bibliographystyle{abbrv}

\end{document}